\documentclass[12pt]{amsart}

\usepackage{amssymb, amsmath, epsfig, graphics, amscd, ifthen, verbatim,diagrams}

\theoremstyle{plain}
\newtheorem{theorem}{Theorem}[section]
\newtheorem{proposition}[theorem]{Proposition}
\newtheorem{corollary}[theorem]{Corollary}

\newtheorem{lemma}[theorem]{Lemma}

\theoremstyle{definition}
\newtheorem{definition}[theorem]{Definition}

\newtheorem{example}[theorem]{Example}

\newtheorem{remark}[theorem]{Remark}

\newcommand{\Spec}{\mathop{\rm Spec}\nolimits}

\newcommand{\Hom}{\mathop{\rm Hom}\nolimits}

\newcommand{\Sym}{\mathop{\rm Sym}\nolimits}
\newcommand{\Hilb}{\mathop{\rm Hilb}\nolimits}

\newcommand{\Exp}{\mathop{\rm Exp}\nolimits}

\newcommand{\GL}{\mathop{\rm GL}\nolimits}

\newcommand{\MHM}{\mathop{\rm MHM}\nolimits}
\newcommand{\Tr}{\mathop{\rm Tr}\nolimits}
\newcommand{\Sing}{\mathop{\rm Sing}\nolimits}

\newcommand{\vol}{\mathop{\rm vol}\nolimits}
\newcommand{\Var}{\mathop{\rm Var}\nolimits}
\newcommand{\relvir}{{\rm relvir}}
\newcommand{\vir}{{\rm vir}}
\newcommand{\red}{{\rm red}}
\newcommand{\aff}{{\rm aff}}
\newcommand{\half}{\frac{1}{2}}

\renewcommand{\Box}{\square}
\renewcommand{\tilde}{\widetilde}

\newcommand{\M}{\mathcal M}

\newcommand{\F}{\mathbb F}

\newcommand{\N}{\mathbb N}

\renewcommand{\AA}{\mathbb A}

\newcommand{\C}{\mathbb C}
\newcommand{\CC}{\mathcal C}
\newcommand{\Z}{\mathbb Z}
\renewcommand{\O}{\mathcal O}

\newcommand{\PP}{\mathbb P}
\newcommand{\MC}{\M_\C}
\newcommand{\CP}{\mathcal P}
\newcommand{\Q}{\mathbb Q}
\renewcommand{\phi}{\varphi}

\newcommand{\LL}{\mathbb{L}}

\newcommand{\Lfact}[1]{[#1]_{\LL }!}
\newcommand{\Lbinom}[2]{\left[\begin{matrix}{#1}\\{#2}\end{matrix}\right]_{\LL }}
\newcommand{\muhat}{{\hat\mu}}
\newcommand{\DT}{Do\-nald\-son--Tho\-mas }
\newcommand{\BiBi}{Bia{\l}ynicki-Birula }
\newcommand{\Index}{\operatorname{index}}
\newcommand{\CY}{Ca\-la\-bi--Yau }
\newcommand{\motivicring}{ring of motivic weights }

\newcommand{\point}{\text{pt}}

\newcommand{\color}[6]{}

\newarrow{Into}C--->

\numberwithin{equation}{section}

\title[Motivic degree zero DT invariants]{Motivic degree zero \\ Donaldson--Thomas invariants}

\author[Kai Behrend]{Kai Behrend}
\address{Dept. of Math., University of British Columbia, Vancouver, BC, Canada.}
\email{behrend@math.ubc.ca}

\author[Jim Bryan]{Jim Bryan}
\address{Dept. of Math., University of British Columbia, Vancouver, BC, Canada.}
\email{jbryan@math.ubc.ca}

\author[Bal\'azs Szendr\H oi]{Bal\'azs Szendr\H oi} 
\address{Math. Inst., University of Oxford, Oxford, United Kingdom.}
\email{szendroi@maths.ox.ac.uk}

\begin{document}

\begin{abstract} Given a smooth complex threefold $X$, we define
the virtual motive $[\Hilb ^{n}(X)]_{\vir }$ of the Hilbert scheme of
$n$ points on $X$. In the case when $X$ is Calabi--Yau, $[\Hilb
^{n}(X)]_{\vir }$ gives a motivic refinement of the $n$-point degree
zero Donaldson--Thomas invariant of $X$. The key example is $X=\C^{3}$, 
where the Hilbert scheme can be expressed as the critical locus of a
regular function on a smooth variety, and its virtual motive is defined
in terms of the Denef--Loeser motivic nearby fiber. A crucial
technical result asserts that if a function is equivariant
with respect to a suitable torus action, its motivic nearby fiber is
simply given by the motivic class of a general fiber. This allows us
to compute the generating function of the virtual motives $[\Hilb ^{n}
(\C ^{3})]_{\vir }$ via a direct computation involving the motivic
class of the commuting variety. We then give a formula for the
generating function for arbitrary $X$ as a motivic exponential,
generalizing known results in lower dimensions. The weight polynomial
specialization leads to a product formula in terms of deformed
MacMahon functions, analogous to G\"ottsche's formula for the
Poincar\'e polynomials of the Hilbert schemes of points on surfaces.
\end{abstract}

\maketitle

%\tableofcontents

\section*{Introduction} 

Let $Z$ be a scheme of finite type over $\C $. The \emph{virtual Euler
characteristic} of $Z$ is defined to be the topological Euler
characteristic, weighted by the integer-valued constructible function
$\nu _{Z}$ introduced by the first author \cite{Behrend-micro}:
\[
\chi _{\vir } (Z) = \sum _{k\in \Z }\,k\,\chi \left(\nu _{Z}^{-1} (k)
\right).
\]
Unlike the ordinary Euler characteristic, the virtual Euler
characteristic is sensitive to singularities and scheme structure.  A
\emph{virtual motive} of $Z$ is an element $[Z]_{\vir }$ in a suitably
augmented Grothendieck group of varieties (the ``\motivicring'', 
see \S\ref{subsec:ring:of:motives} $\MC$) such that
\[
\chi \left([Z]_{\vir} \right) = \chi _{\vir } (Z).
\]
In the context where $Z$ is a moduli space of sheaves on a \CY
threefold $X$, the virtual Euler characteristic $\chi _{\vir } (Z)$ is
a (numerical) \DT invariant. In this setting, we say that $[Z]_{\vir
}$ is a \emph{motivic \DT invariant}.

In this paper, we construct a natural virtual motive $[\Hilb ^{n}
(X)]_{\vir }$ for the Hilbert scheme of~$n$ points on a smooth
threefold $X$. In the \CY case, the virtual Euler characteristics of 
the Hilbert schemes of points are the degree zero \DT invariants of $X$ 
defined in \cite{MNOP1} and computed in \cite{Behrend-Fantechi08,Li-zeroDT,
Levine-Pandharipande} (cf.~Remark~\ref{rem:CY_or_not_CY}). 
So we call our virtual motives $[\Hilb ^{n}
(X)]_{\vir }$ the motivic degree zero \DT invariants of~$X$.

If $f:M\to \C $ is a regular function on a smooth variety and
\[
Z=\{df=0 \}
\]
is its scheme theoretic degeneracy locus, then there is a natural
virtual motive (Definition~\ref{defn: virtual motive}) given by
\[
[Z]_{\vir } = -\LL ^{-\frac{\dim M}{2}}[\phi _{f}],
\]
where $[\phi _{f}]$ is the motivic vanishing cycle defined by
Denef-Loeser \cite{Denef-Loeser-geometry,Looijenga-motivic} and $\LL$
is the Lefschetz motive. (This is similar to Kontsevich and Soibelman's
approach to motivic Donaldson--Thomas invariants
\cite{Kontsevich-Soibelman}). The function $f$ is often called a global
Chern-Simons functional or super-potential. This setting
encompasses many useful cases such as when $Z$ is smooth (by letting
$(M,f)= (Z,0)$) and (less trivially) when $Z$ arises as a moduli space
of representations of a quiver equipped with a super-potential.  The
latter includes $\Hilb^{n} (\C ^{3})$ which we show is given as the
degeneracy locus of an explicit function $f_{n}:M_{n}\to \C$ on a
smooth space $M_n$ (see \S\ref{subsec: the hilb scheme as a crit
locus}).

We prove (Theorem~\ref{thm: main result of appendix},
cf. Propositions~\ref{prop: if f is proper and C* equivariant then
vanishing cycle is X1-X0}, \ref{prop: relative and effective version
of vanishing cycle = X1 - X0 proposition}, and \ref{prop: if f is C*
equiv and circle compact, then vanishing cycle is X1-X0}) that if
$f:M\to \C $ is equivariant with respect to a torus action satisfying
certain properties, then the motivic vanishing cycle is simply given
by the class of the general fiber minus the class of the central
fiber:
\[
[\phi _{f}] = [f^{-1} (1)] - [f^{-1} (0)].
\]
This theorem should be applicable in a wide variety of quiver settings
and should make the computation of the virtual motives $[Z]_{\vir }$
tractable by quiver techniques.

Indeed, we apply this to compute the motivic degree zero \DT partition
function
\[
Z_{X} (t) = \sum _{n=0}^{\infty } [\Hilb ^{n} (X)]_{\vir }\,t^{n}
\]
in the case when $X$ is $\C ^{3}$. Namely, it is given in
Theorem~\ref{thm_C3} as
\[
Z_{\C ^{3}} (t) = \prod _{m=1}^{\infty }\prod _{k=0}^{m-1}\left(1-\LL ^{k+2-\frac{m}{2}}\,t^{m} \right)^{-1}.
\]

The virtual motive $[\Hilb ^{n} (\C ^{3})]_{\vir}$ that we construct
via the super-potential $f_{n}$ has good compatibility properties with
respect to the Hilbert-Chow morphism $\Hilb ^{n} (\C ^{3})\to \Sym
^{n} (\C ^{3})$.  Consequently we are able to use these virtual
motives to define virtual motives $[\Hilb ^{n} (X)]_{\vir }$ for the
Hilbert scheme of any smooth threefold $X$ (see \S\ref{subsec: virtual
motives of the Hilb scheme of a general threefold}). The now standard
technology~\cite{Cheah, Gottsche-motive,
Gusein-Zade-Luengo-Melle-Hernandez-Hilb,Getzler-mixed-hodge} allows us
to express in Theorem~\ref{thm_general_formula} the motivic degree
zero \DT partition function of any threefold $X$ as a motivic
exponential
\[
Z_{X} (-t) = \Exp \left([X] \frac{-\LL ^{-\frac{3}{2}}\,t}{(1+\LL
^{\frac{1}{2}}\,t)(1+\LL ^{-\frac{1}{2}}\,t)} \right).
\]
See \S\ref{subsec: power str} for the definition of $\Exp $.

While the above formula only applies when $\dim (X)=3$, it fits well
with corresponding formulas for $\dim (X)<3$. In these cases, the
Hilbert schemes are smooth and thus have canonical virtual motives
which are easily expressed in terms of the ordinary classes $[\Hilb
^{n} (X)]$ in the Grothendieck group. The resulting partition
functions have been computed for curves
\cite{Gusein-Zade-Luengo-Melle-Hernandez-Hilb} and surfaces
\cite{Gottsche-motive}, and all these results can be expressed
(Corollary~\ref{cor: pithy formula for Z in all dims}) in the
single formula 
\[
Z_{X} (T) = \Exp \Big (\,T [X]_{\vir } \Exp \left(T[\PP ^{d-2}]_{\vir
} \right) \, \Big )
\]
valid when $d=\dim (X)$ is 0, 1, 2, or 3. Here 
\[  [X]_{\vir } = \LL^{-\frac{d}{2}}[X],  
\]
also\footnote{Note that the
variable ``$t$'' has special meaning in the definition of ``$\Exp$'';
in particular, one cannot simply substitute $t$ for $T$ in the above
equation for~$Z_{X}$.}
\[
T= (-1)^{d}t,
\]
and the class of a negative dimensional projective space is defined
by~\eqref{eqn: defn of neg dim proj space}. In particular, $[\PP
^{-1}]_{\vir } = 0$ and $[\PP ^{-2}]_{\vir } = -1$. 
There are some indications that the above
formula has significance for $\dim X >3$; see Remarks \ref{rem: the
one formula is valid for n<4} and \ref{rem: Euler char specialization
of the one formula is MacMahon's guess for dim d partitions}.

The weight polynomial specialization of the class of a projective
manifold gives its Poincar\'e polynomial. For example, if $X$ is a
smooth projective threefold, we get
\[
W \left([X],q^{\frac{1}{2}} \right) = \sum _{d=0}^{6} b_{d}\,q^{\frac{d}{2}},
\]
where $b_{d}$ is the degree $d$ Betti number of $X$.  Taking the
weight polynomial specialization of the virtual motives of the Hilbert
schemes gives a virtual version of the Poincar\'e polynomials of the
Hilbert schemes. In Theorem~\ref{thm: formula for virtual weight
polys} we apply weight polynomials to our formula for $Z_{X} (t)$ to
get
\[
\sum _{n=0}^{\infty } W \left([\Hilb ^{n} (X)]_{\vir },q^{\frac{1}{2}}
\right)\,t^{n} = \prod _{d=0}^{6} M_{\frac{d-3}{2}}\left(-t,-q^{\frac{1}{2}}
\right)^{(-1)^{d}b_{d}}
\]
where $M_{\delta }$ is the $q$-deformed MacMahon
function
\[
M_{\delta } \left(t,q^{\frac{1}{2}} \right) = \prod _{m=1}^{\infty
}\prod _{k=0}^{m-1} \left(1- q^{k+\frac{1}{2}-\frac{m}{2}+\delta
}\,t^{m} \right)^{-1}.
\]
The above formula is the analog for threefolds of G\"ottsche's famous
product formula \cite{Gottsche-formula} for the Poincar\'e polynomials of
Hilbert schemes on surfaces. Similar $q$-deformed MacMahon functions
appear in the refined topological vertex of Iqbal-Kozcaz-Vafa
\cite{Iqbal-Kozcaz-Vafa}. We discuss $M_{\delta }$ further in
Appendix~\ref{appendix: q-def of MacMahon}.

In addition to motivic \DT invariants, one may consider
\emph{categorified \DT invariants}. Such a categorification is a lift
of the numerical \DT invariant $\chi _{\vir } (Z)$ to an object
$[Z]_{cat}$ in a category with a cohomological functor $H^{\bullet }$
such that 
\[
\chi (H^{\bullet } ([Z]_{cat}))=\chi _{\vir } (Z).
\]
We have partial results toward constructing categorified degree zero
\DT invariants which we discuss in \S\ref{subsec: categorified DT
invs}. See also the recent work~\cite{KS_new}.

\section{Motivic weights and vanishing cycles}

All our varieties and maps are defined over the field $\C$ of complex
numbers.

\subsection{The \motivicring}
\label{subsec:ring:of:motives}

Let $K_0(\Var_\C)$ be the $\Z$-module generated by isomorphism classes
of reduced $\C$-varieties\footnote{One can also consider the rings
$K_{0} (Sch_{\C })$ and $K_{0} (Sp_{\C})$ generated by schemes or
algebraic spaces (of finite type) with the same relations. By
\cite[Lemma~2.12]{Bridgeland-intro}, they are all the same. In
particular, for $X$ a scheme, its class in $K_{0} (\Var _{\C })$ is
given by $[X_{\red }]$, the class of the associated reduced scheme. We
will implicitly use this identification throughout the paper without
further comment.}, under the scissor relation
\[  [X]=[Y] + [X\setminus Y]\in K_0(\Var_\C)
\]
for $Y\subset X$ a closed subvariety. $K_0(\Var_\C)$ has a ring
structure whose product is the Cartesian product of varieties. The
following two properties of the setup are well known.

\begin{enumerate}
\item If $f\colon X\to S$ is a Zariski locally trivial fibration with 
fiber $F$, then
\[ [X]=[S]\cdot[F]\in K_0(\Var_\C).
\]
\item If $f\colon X\to Y$ is a bijective morphism, then 
\[ [X]=[Y]\in K_0(\Var_\C).
\]
\end{enumerate}

Let 
\[ \LL =[\AA^1]\in K_0(\Var_\C)\]
be the class of the affine line. We define the \emph{\motivicring }
(or \emph{motivic ring} for short) to be
\[  
\M_\C= K_0(\Var_\C)[\LL ^{-\half}]. 
\]

We set up notation for some elements of $\M_\C$ that we will use
later. Let
\[
\Lfact{n} = (\LL ^{n}-1) (\LL ^{n-1}-1)\dotsb (\LL -1)
\]
and let
\[
\Lbinom{n}{k} = \frac{\Lfact{n}}{\Lfact{n-k}\Lfact{k}}.
\]
Using property (1) above, elementary arguments show that
\[
[\GL_{n}] = \LL ^{\binom{n}{2}}\Lfact{n}
\]
Then using the elementary identity
$\binom{n}{2}-\binom{k}{2}-\binom{n-k}{2}= (n-k)k$, the class of the
Grassmanian is easily derived:
\[
\left[Gr (k,n) \right] =
\Lbinom{n}{k}.
\]
For later reference, we recall the computation of the motivic weight
of the stack of pairs of commuting matrices. Let $V_n$ be an
$n$-dimensional vector space, and let \[C_n\subset \Hom(V_n,
V_n)^{\times 2}\] denote the (reduced) variety of pairs of commuting
$n\times n$ matrices over the complex numbers. Let
\begin{equation}
\tilde{c}_{n}= \frac{[C_{n}]}{[\GL _{n}]}\in\M_\C[(1-\LL ^n)^{-1}\colon {n\geq 1}]
\label{eq:tildecn}
\end{equation}
be its class, renormalized by taking account of the global symmetry 
group $\GL_n$. Consider the generating series
\begin{equation} C(t)= \sum_{n\geq 0} \tilde c_n t^n. 
\label{eq_Cseries}
\end{equation}
\begin{proposition} \label{prop: Feit-Fine}
We have
\begin{equation}
C(t) = \prod_{m=1}^{\infty }\prod_{j=0}^{\infty} (1-\LL ^{1-j}t^{m})^{-1}.
\label{eq_Cseries_result}
\end{equation}
\end{proposition}
\proof The main result of the paper of Feit and Fine \cite{Feit-Fine}
is the analogous formula
\[
C (t,q) = \prod_{m=1}^{\infty }\prod_{j=0}^{\infty} (1-q^{1-j}t^{m})^{-1}
\]
for the generating series of the number of pairs of commuting matrices
over the finite field~$\F_q$, renormalized as above. Feit and Fine's
method is motivic; in essence they provide an affine paving of
$C_{n}$. For details, see~\cite{AMorrison}. \endproof

\begin{remark}
\label{rem: coeffs of the feit-fine formula are rational fncs in L}
In~\eqref{eq_Cseries}, the coefficient $\tilde c_n$ of $t^{n}$ in $C(t)$ is 
in the ring $\M_\C[(1-\LL ^n)^{-1}\colon {n\geq 1}]$. In~\eqref{eq_Cseries_result}, 
the coefficients are the Laurent expansions in~$\LL $ of these elements.
\end{remark}

The following result is now standard~\cite[Lemma 4.4]{Gottsche-motive}.

\begin{lemma} Let $Z$ be a variety with the free action of a finite 
group~$G$. Extend the action of $G$ to $Z\times\AA^n$ using a linear action 
of $G$ on the second factor. Then the motivic weights of the quotients
are related by
\[ [(Z\times\AA^n)/G] = \LL^n[Z/G]\in\MC. 
\]
\label{lem}
\end{lemma}
\begin{proof} Let $\pi\colon (Z\times \AA^n)/G \to Z/G$ be the projection. 
By assumption, $\pi$ is \'etale locally trivial with fiber $\AA^n$, and with 
linear transition maps. Thus it is an \'etale vector bundle on $Z/G$. 
But then by Hilbert's Theorem 90, it is Zariski locally trivial. 
\end{proof} 

\subsection{Homomorphisms from the \motivicring}

The \motivicring admits a number of well-known ring homomorphisms.
Deligne's mixed Hodge structure on compactly supported cohomology of a 
variety~$X$ gives rise to the E-polynomial homomorphism
\[ E\colon K_0(\Var_\C)\to \Z[x,y]\]
defined on generators by
\[ E([X];x,y) = \sum_{p,q} x^p y^q \sum_{i} (-1)^i \dim H^{p,q}(H^i_c(X,\Q)).\]
This extends to a ring homomorphism
\[  E\colon \M_\C\to \Z[x,y,(xy)^{-\half }]\]
by defining 
\[  E(\LL ^n)= (xy)^{n}\]
for half-integers~$n$. 

The weight polynomial homomorphism
\[ W\colon \M_\C\to\Z[q^{\pm\half}],\] 
is defined by the specialization
\[
x=y=-q^{\half },\quad (xy)^{\half } = q^{\half }.
\]
This maps $\LL $ to $q$, and encodes the dimensions of the graded
quotients of the compactly supported cohomology of $X$ under the
weight filtration, disregarding the Hodge filtration.  For smooth
projective~$X$, $W([X]; q^\half)$ is simply the Poincar\'e polynomial
of $X$. Specializing further,
\[ \chi([X]) = W([X]; q^\half=-1) \]
defines the map
\[\chi\colon \M_\C\to \Z\]
of compactly supported Euler characteristic;
this agrees with the ordinary Euler characteristic. 

\subsection{Relative motivic weights} Given a reduced (but not
necessarily irreducible) variety~$S$, let $K_0(\Var_S)$ be the
$\Z$-module generated by isomorphism classes of (reduced)
$S$-varieties, under the scissor relation for $S$-varieties, and ring
structure whose multiplication is given by fiber product over
$S$. Elements of this ring will be denoted $[X]_S$.  A morphism
$f:S\to T$ induces a ring homomorphism $f^*\colon K_0(\Var_T)\to
K_0(\Var_S)$ given by fiber product. In particular, $K_0(\Var_S)$ is
always a $K_0(\Var_\C)$-module. Thus we can let
\[
\M_S = K_0(\Var_S)[\LL ^{-\half}], 
\]
an $\M_\C$-module. A morphism $f:S\to T$ induces a ring homomorphism
$f^*\colon\M_T\to\M_S$ by pullback, as well as a direct image
homomorphism $f_!\colon \M_S\to \M_T$ by composition, the latter a map
of $\M_T$-modules.

In the relative case, the E-polynomial, weight polynomial and Euler
characteristic specializations map to the K-group of variations of
mixed Hodge structures (or mixed Hodge modules), the K-group of mixed
sheaves, and the space of constructible functions, respectively.

\subsection{Equivariant motivic weights}

Let $G$ be a finite group. An action of $G$ on a variety $X$ is said
to be good, if every point of $X$ is contained in an affine
$G$-invariant open subset; all actions in this paper are going to be
good. 

We will have occasion to use two versions of equivariant rings of
motivic weights.
For a fixed variety $S$ with $G$-action, let $\tilde K^G_0(\Var_S)$ be
the K-group generated by $G$-varieties over $S$, modulo the
$G$-scissor relation. 
Let also $K^G_0(\Var_S)$ be the quotient of $\tilde K^G_0(\Var_S)$ by
the further relations
\begin{equation}\label{eqn: vector bundle with G action is trivial in equivariant K-group}
[V,G] = [\C ^{r}\times S]
\end{equation}
where $V\to S$ is any $G-$equivariant
vector bundle over $S$ of rank $r$ and $\C ^{r}\times S$ is the 
trivial rank $r$ bundle with trivial $G$-action. The affine $S$-line
$\AA^1\times S$ inherits a $G$-action and so defines elements $\LL
\in \tilde K^G_0(\Var_S)$ and $\LL \in K^G_0(\Var_S)$; 
we let $\tilde\M^G_S$ and $\M^G_S$ be the corresponding extensions by $\LL^{-\half}$.

If the $G$-action on $S$ is trivial, then product makes 
$\M^G_S$ and $\tilde \M^G_S$ into $\M_S$-algebras.
In this case, there is a map of $\M_S$-modules
\begin{equation}\label{eq: quotient map on rings} 
\pi_G\colon \tilde\M^{G}_S \to \M_S
\end{equation} 
given on generators by taking the orbit space. This operation is
clearly compatible with the module operations and scissor relation. In
general, this map does not respect the relations~\eqref{eqn: vector
bundle with G action is trivial in equivariant K-group}, so it does
not descend to $\M^{G}_S$.

As a variant of this construction, let $\hat\mu=\lim_{\leftarrow}
\mu_n$ be the group of roots of unity. A good $\hat\mu$-action on a
variety $X$ is one where $\hat\mu$ acts via a finite quotient and that
action is good.  Let $\M^\muhat_S=K_0^\muhat(\Var_S)[\LL^{-\half}]$ be the
corresponding $K$-group, incorporating 
the relations~\eqref{eqn: vector bundle with G action is trivial in equivariant K-group}.
The additive group $\M_S^\muhat$ can be endowed with an 
associative operation $\star$ using convolution
involving the classes of Fermat curves~\cite{Denef-Loeser-Igusa, Looijenga-motivic, KS_new}. 
This product agrees with the ordinary (direct) product on the subalgebra
$\M_S\subset\M_S^\muhat$ of classes with trivial $\muhat$-action, but not in general.

We will need the following statement below.

\begin{lemma} Let $A\in \MC$ be a general class, $n$ a positive 
integer. Then the class $A^n\in\MC$ can be given a structure of a 
class in $\tilde \M^{S_n}_\C$, the ring of $S_n$-equivariant 
motivic weights. \label{lemma: symm grp acts}
\end{lemma}
\begin{proof} Up to powers of $\LL^{\pm\half}$, we can 
write $A=B-C$ where $B, C$ are classes represented by 
actual varieties. Then by the binomial theorem,
\[ A^n = \sum_{i=0}^n (-1)^i X_i,\]
where $X_i$ is the variety which consists of $n \choose i$ disjoint copies of 
the variety $B^{n-i} C^i$. We claim that all these varieties $X_i$ carry 
geometric $S_n$-actions. Label the $B$'s and $C$'s in the 
expansion of $(B-C)^n$ with the labels $1, \ldots, n$, depending on which 
bracket they come from. Then every term will carry exactly one instance of 
each label. The group~$S_n$ acts by interchanging the labels, and does not 
change the number of $B$'s and $C$'s in a monomial. Thus each element 
of $S_n$ defines a map $X_i\to X_i$, i.e. a geometric automorphism of $X_i$. 
\end{proof}

\subsection{Power structure on the \motivicring}\label{subsec: power
str} Recall that a power structure on a ring $R$ is a map
\begin{align*} (1 + t R[[t]])\times R & \to 1  + t R[[t]]\\
(A(t),m) &\mapsto A(t)^m
\end{align*}
satisfying $A(t)^0 = 1$, $A(t)^1=A(t)$, 
$A(t)^{m+n}=A(t)^mA(t)^n$, $A(t)^{mn}=(A(t)^m)^n$,
$A(t)^m B(t)^m = (A(t)B(t))^m$, as well as $(1+t)^m = 1+mt + O(t^2)$. 

\begin{theorem}[Gusein-Zade et
al~\cite{Gusein-Zade-Luengo-Melle-Hernandez-power}, cf. Getzler
\cite{Getzler-mixed-hodge}] There exists a power structure on the
Grothen\-dieck ring $K_0(\Var_\C)$, defined uniquely by the property
that for a variety~$X$,
\[ (1-t)^{-[X]} = \sum_{n=0}^{\infty } [\Sym^n X]t^n
\]
is the generating function of symmetric products of $X$, its motivic 
zeta function. 
\label{theorem power structure}
\end{theorem}

Since we will need it below, we recall the definition. 
Let $A(t)=1+\sum_{i\geq 1} A_i t^i$ be a series with $A_i\in K_0(\Var_\C)$. 
Then for $[X]\in K_0(\Var_\C)$ a class represented by a variety~$X$, 
the definition of~\cite{Gusein-Zade-Luengo-Melle-Hernandez-power} reads

\begin{equation}A(t)^{[X]} = 1 + \sum_{\alpha}\pi_{G_\alpha}\left[\left(\prod_iX^{\alpha_i}\setminus \Delta\right)\cdot \left(\prod_i A_i^{\alpha_i}\right)\right] t^{|\alpha|}.
\label{def power structure}
\end{equation}
Here the summation runs over all partitions $\alpha$; 
for a partition $\alpha$, write $\alpha_i$ for the number of parts of 
size~$i$, and let $G_\alpha=\prod_i S_{\alpha_i}$ denote the standard product 
of symmetric groups. $\Delta$ denotes the big diagonal in any product 
of copies of the variety $X$. By Lemma~\ref{lemma: symm grp acts}, the product 
$\left(\prod_iX^{\alpha_i}\setminus \Delta\right)\times \prod_i A_i^{\alpha_i}$ 
can be represented by a class in $K_0^{G_\alpha}(\Var_\C)$, and
the map $\pi_{G_\alpha}$ is the quotient map~\eqref{eq: quotient map on rings}. 

Note that if we replace the coefficients $A_i$ by $\LL^{c_i} A_i$ for positive
integers $c_i$, then by Lemma~\ref{lem} above, the individual terms in the 
sum change as
\[\begin{array}{c} \displaystyle\pi_{G_\alpha}\left[\left(\prod_iX^{\alpha_i}\setminus \Delta\right)\times \prod_i (\LL^{c_i}A_i)^{\alpha_i}\right]
= \\ \ \ \ \ \ \ \ \ \ \ \ \ \ = \displaystyle\LL^{\sum_i c_i\alpha_i}\pi_{G_\alpha}\left[\left(\prod_iX^{\alpha_i}\setminus \Delta\right)\times \prod_i A_i^{\alpha_i}\right],
\end{array}
\]
since the action of $G_\alpha$ on the $\LL^{\sum_i c_i\alpha_i}$ factor comes 
from a product of permutation actions and is hence linear. We thus get 
the substitution rule
\[ A(\LL^c t)^{[X]} = A(t)^{[X]}\Big|_{t\mapsto \LL^c t}
\]
for positive integer $c$. We extend the definition~\eqref{def power structure}
to allow coefficients $A_i$ which are from $\MC$, 
by the formula 
\[\begin{array}{c} \displaystyle\pi_{G_\alpha}\left[\left(\prod_iX^{\alpha_i}\setminus \Delta\right)\times \prod_i ((-\LL^\half)^{c_i}A_i)^{\alpha_i}\right]
= \\ \ \ \ \ \ \ \ \ \ \ \ \ \ = \displaystyle (-\LL^\half)^{\sum_i c_i\alpha_i}\pi_{G_\alpha}\left[\left(\prod_iX^{\alpha_i}\setminus \Delta\right)\times \prod_i A_i^{\alpha_i}\right]
\end{array}\]
for integers $c_i$ (see Remark~\ref{rem: lambda ring homo} for the reason for
the appearance of signs here). 
This implies the substitution rule
\[ A\left((-\LL^\half)^n t\right)^{[X]} = A(t)^{[X]}\Big|_{t\mapsto (-\LL^\half)^n t}
\]
for integers $n$.
We also extend the power structure to exponents from the
ring~$\MC$ by defining
\[(1-t)^{-\left(-\LL ^{\half } \right)^{n}[X]} = \left(1-\left(-\LL^{\half } \right)^{n} t \right)^{-[X]}\]
for all $n\in \Z $; as in Theorem~\ref{theorem power structure}, this 
determines a unique extension of the power structure. 

Finally, still
following~\cite{Gusein-Zade-Luengo-Melle-Hernandez-power,
Getzler-mixed-hodge}, introduce the map
\[ 
\Exp\colon t\MC[[t]]\to 1 +t\MC[[t]]
\]
by
\[ 
\Exp\sum_{n= 1}^{\infty }[A_n]t^n = \prod_{n\geq 1} (1-t^n)^{-[A_n]}.
\]
Note that the variable~$t$ plays a special role here. 

Note that the above equations imply the following substitution 
rule:
\begin{equation}\label{eqn: substitution rule for Exp}
\Exp (A (t))|_{t\mapsto \left(-\LL ^{\half } \right)^{n} t} =\Exp
\left(A\left(\left(-\LL ^{\half } \right)^{n}t \right) \right).
\end{equation}

\begin{example} It is easy to check that the generating series $C(t)$
of the motivic weight of pairs of commuting
matrices~\eqref{eq_Cseries_result} can be written as a motivic
exponential (cf. \cite[Prop.~7]{KS_new}):
\[
C(t) = \Exp \left(\frac{\LL ^{2}}{\LL -1}\,\frac{t}{1-t} \right).
\]
\end{example}

\begin{remark}\label{rem: lambda ring homo}
The existence a power structure on $K_{0} (\Var _{\C })$ is closely
related to the fact that $K_{0} (\Var _{\C })$ has the structure of a
pre-$\lambda $-ring where the operations $\sigma _{n}$ are
characterized by $\sigma _{n} (X) = [\Sym ^{n} (X)]$ (see
\cite{Getzler-mixed-hodge}). In order to extend the power structure to~$\MC$
so that the Euler characteristic homomorphism
respects the power structure, we must have\footnote{We thank Sven Meinhardt for calling
our attention to this sign issue.} $\sigma
_{n} \left(-\LL ^{\half } \right) = \left(-\LL ^{\half } \right)^{n}$
which explains the signs in the formulae above. 
\end{remark}

\subsection{Motivic nearby and vanishing cycles}

Let 
\[
f\colon X\to\C
\]
be a regular function on a smooth
variety~$X$, and let $X_0=f^{-1}(0)$ be the central fiber. 

Using arc spaces, Denef and Loeser
\cite[\S~3]{Denef-Loeser-geometry},\cite[\S~5]{Looijenga-motivic}
define $[\psi _{f}]_{X_{0}}\in \M _{X_{0}}^{\muhat }$, the relative
motivic nearby cycle of $f$. Using motivic integration,
Denef--Loeser give an explicit formula for $[\psi _{f}]$ in terms of an
embedded resolution of $f$. We give this formula in detail in
Appendix~\ref{appendix: nearby fiber proof}.

Let
\[ 
[\phi_f]_{X_0} = [\psi_f]_{X_0} - [X_0]_{X_0}\in \M^\muhat_{X_0}
\]
be the relative motivic vanishing cycle of~$f$. It follows directly
from the definitions that over the smooth locus of the central fiber,
the classes $[\psi_f]_{X_0}$ and $[X_0]_{X_0}$ coincide, so the
motivic difference $[\phi_f]_{X_0}$ is a relative class
$[\phi_f]_{\Sing(X_0)}$ over the singular locus of $X_0$. The latter
is exactly the degeneracy locus $Z\subset X$, the subscheme of $X_0$
corresponding to the subscheme given by the equations $\{df=0\}$. We
will denote by $[\psi_f], [\phi_f]\in\M^{\muhat }_{\C}$ the absolute
motivic nearby and vanishing cycles, the images of the relative
classes under pushforward to the point.

We next recall the motivic Thom--Sebastiani theorem. Given two regular
functions $f\colon X\to\C$ and $g\colon Y\to\C$ on smooth varieties
$X, Y$, define the function $f\oplus g\colon X\times Y\to\C$ by
\[ 
(f\oplus g)(x,y)=f(x)+g(y).
\]

\begin{theorem} {\rm (Denef--Loeser~\cite{Denef-Loeser-Thom}, 
Looijenga \cite{Looijenga-motivic})} 
Let $f,g$ be non-constant regular functions 
on smooth varieties $X,Y$, and let $X_0, Y_0$ be their zero fibers. Let
\[ 
i\colon X_0\times Y_0 \to (X\times Y)_0
\]
denote the natural inclusion into the zero fiber of $f\oplus g$. Then
\[ 
i^*[-\phi_{f\oplus g}]_{X_{0}\times Y_{0}} = p_X^*[-\phi_f]_{X_{0}}
\star p_Y^*[-\phi_g]_{Y_{0}}\in\M^\muhat_{X_0\times Y_0}
\]
where $p_X, p_Y$ are the projections from $X_0\times Y_0$ to the two factors. 
\label{thm_thom_seb}
\end{theorem}
\begin{remark}\label{rem: 1-[2pt,mu2] is a sqrt of L}
Consider the functions $f (x)=x^{2}$ and $g (y)=y^{2}$. Restricting to
the origins in $\C $ and $\C ^{2}$, Theorem~\ref{thm_thom_seb} reads
\[
(-\phi _{x^{2}}) \star (-\phi _{y^{2}}) = -\phi _{x^{2}+y^{2}} \in \M
^{\muhat }_{\C }.
\]
Direct computation (using for example~\eqref{eqn:
Denef-Loeser defn of motive of nearby cycles}) yields
\begin{align*}
-\phi _{x^{2}+y^{2}}&= \LL &-\phi _{x^{2}} &=-\phi
_{y^{2}}=1-[2\point ,\mu _{2}]
\end{align*}
where $[2\point ,\mu _{2}]$ is the space of 2 points, with the $\mu
_{2}$-action which swaps the points. We see that rather than adjoining
$\LL ^{\half }$ formally to $\M ^{\muhat }_{\C }$, we could have taken
\begin{equation}\label{eqn: sqrt (L) = 1-[2pt,mu2])}
\LL ^{\half } = 1-[2\point ,\mu _{2}]
\end{equation}
(cf. \cite[Remark 19]{Kontsevich-Soibelman}). Indeed, imposing the
above equation as a relation in $\M ^{\muhat }_{\C }$ has some
desirable consequences such as making the relative virtual motive of a
smooth variety canonical at each point; see Remark~\ref{rem: the case
when Z= (df=0) is smooth}.\footnote{We thank Sheldon Katz for
discussions on this issue.}
\end{remark}

\subsection{Torus-equivariant families}

We wish to study regular functions $f\colon X\to\C$ on smooth
varieties~$X$ with the following equivariance property\footnote{We
thank Patrick Brosnan and J\"org Sch\"urmann for very helpful
correspondence on this subject.}\label{subsec:triv}.  We assume
there exists an action of a connected complex torus $T$ on $X$ such
that $f$ is $T$-equivariant with respect to a primitive character
$\chi :T\to \C ^{*}$. That is, for all $t\in T$ and $x\in X$, we have
$f (tx)=\chi (t)f (x)$.

Assuming the existence of such a $T$-action, the family defined by $f$
is trivial away from the central fiber. Indeed, since $\chi $ is
primitive, there exists a 1-parameter subgroup $\C ^{*}\subset T$ such
that $\chi $ is an isomorphism restricted to $\C ^{*}$.  Let
$X_{1}=f^{-1} (1)$, then the map $X_{1}\times \C ^{*}\to X-X_{0}$
given by $(x,\lambda )\mapsto \lambda \cdot x$ has inverse
\[
x\mapsto \left(\frac{1}{f (x)}\cdot x,f (x) \right)
\]
and thus defines an isomorphism $X_{1}\times \C ^{*}\cong
X-X_{0}$.

\begin{proposition}\label{prop: if f is proper and C* equivariant then
vanishing cycle is X1-X0} Assume that the regular function $f\colon
X\to\C$ on a smooth variety~$X$ has a $T$-action as above, and assume
that $f$ is proper.  Then the absolute motivic vanishing cycle 
$[\phi_f]$ of~$f$ lies in the subring
$\M_\C\subset\M_\C^\muhat$. Moreover, this class can be computed as
the motivic difference
\[  [\phi_f]=[X_1]-[X_0]\in\M_\C
\]
of the general and central fibers of~$f$.
\end{proposition}
\proof Using the trivialization of the family discussed above, there
is a diagram
\[\begin{diagram}[height=1.8em,width=1.8em,nohug] X && \rDashto && X_1\times\C\\ & \rdTo_{ f\,\,} && \ldTo_{p} &\\  && \C&& \end{diagram}\]
with the birational map being an isomorphism over $\C^*$; here~$X_t$
denotes the fiber of~$f$ over $t\in\C$, and~$p$
denotes the projection to the second factor.

The fiber product~$W$ of $f$ and $p$ is proper over $\C$. Let $\bar Z$ 
be the irreducible component of the closure of the graph $\Gamma_g$ which maps
dominantly to $\C$; $\bar Z$ is proper over the fiber product $W$ 
so proper and birational over $X$ and $X_1\times\C$. 
Let $Z$ be a desingularization of $\bar Z$. We get a diagram
\[\begin{diagram}[height=1.8em,width=1.8em,nohug] && Z &&\\ & \ldTo^{g} && \rdTo^{h} \\ X && \rDashto && X_1\times\C\\ & \rdTo_{f\,\,} && \ldTo_{p} &\\  && \C&& \end{diagram} \]
with $g, h$ proper maps. 

Denote the composite $f\circ g=p_2\circ h$ by~$k$. 
On the central fibers, we get a diagram
\[ \begin{diagram}[height=1.8em,width=1.8em,nohug] && Z_{0} &&\\ & \ldTo^{g_{0}} && \rdTo^{h_{0}} \\  X_{0} &&  && X_1\\ & \rdTo_{f_{0}} && \ldTo_{p_{0}} &\\ && \{0 \}&& \end{diagram} \]
since the central fiber of the family~$p$ is $X_1$; 
of course the central fibers are no longer birational necessarily.

By \cite[Rem~2.7]{Bittner-motivic}, for the motivic relative nearby
cycles,
\[ [\psi_{f}]_{X_0} = g_{0!}[\psi_{k}]_{Z_0}\in\M_{X_0}^\muhat, 
\]
and 
\[ [\psi_{p}]_{X_1} = h_{0!}[\psi_{k}]_{Z_0}\in\M_{X_1}^\muhat.
\]
Thus, the absolute motivic nearby cycle of $f$ is given by
\[ 
[\psi_{f}] = f_{0!} g_{0!}[\psi_{k}]_{Z_0}= k_{0!}[\psi_{k}]_{Z_0} =
p_{0!} [\psi_{p}]_{X_1}\in\M_\C^\muhat.
\]
But $p$ is an algebraically trivial proper family over $\C$, so its
motivic nearby cycle is the class of its central fiber with trivial
monodromy.  So the absolute motivic nearby cycle of~$f$ is
\[ 
[\psi_{f}] = p_{0!} [ X_1]_{X_1} = [ X_1] \in\M_\C\subset \M_\C^\muhat.
\]
Finally by definition,
\[ [\phi_f]=[\psi_f]-[X_0], 
\]
with $X_0$ carrying the trivial $\muhat$-action. The proof is complete.
\endproof

In our examples, our~$f$ will not be proper. To weaken this
assumption, we say that an action of $\C^*$ on a variety~$V$ is
\emph{circle compact}, if the fixed point set $V^{\C^{*}}$ is compact
and moreover, for all $v\in V$, the limit $\lim_{\lambda\to
0}\lambda \cdot y$ exists. We use the following variant of
Proposition~\ref{prop: if f is proper and C* equivariant then
vanishing cycle is X1-X0}.

\begin{proposition}\label{prop: if f is C* equiv and circle compact, then vanishing cycle is X1-X0}
Let $f:X\to \C $ be a regular function on a smooth quasi-projective
complex variety. Suppose that $T$ is a connected complex torus with a
linearized action on $X$ such that $f$ is $T$-equivariant with respect
to a primitive character $\chi $, i.e. $f (tx)=\chi (t)f (x)$ for all
$t\in T$ $x\in X$. Moreover, suppose that there exists $\C ^{*}\subset
T$ such that the induced $\C ^{*}$-action on $X$ is circle
compact. Then the absolute motivic vanishing cycle $[\phi_f]$ of~$f$
lies in the subring $\M_\C\subset\M_\C^\muhat$, and it
can be expressed as the motivic difference
\[  
[\phi_f]=[X_1]-[X_0]\in\M_\C
\]
of the general and central fibers of~$f$.
\end{proposition}
We do not have a conceptual proof of this Proposition as we did for
Proposition~\ref{prop: if f is proper and C* equivariant then
vanishing cycle is X1-X0}. Instead, in Appendix~\ref{appendix: nearby
fiber proof}, we prove this directly using Denef and Loeser's motivic
integration formula for $[\phi _{f}]$. The key point is that the
circle compact $\C ^{*} $-action gives rise to a \BiBi stratification
of $X$. The conditions that $X$ is quasi-projective and the $T$-action
is linear can probably weakened; they are added for convenience in the
proof and because they should hold in most cases of interest. The
condition that the fixed point set of the $\C ^{*}$-action is compact
can be dropped; we only use the existence of $\lambda \to 0 $ limits.

We will also use the following enhancement of the proposition.
\begin{proposition}\label{prop: relative and effective version of vanishing cycle = X1 - X0 proposition}
Let $f:X\to \C $ be a $T$-equivariant regular function satisfying the
assumptions of Proposition~\ref{prop: if f is C* equiv and circle
compact, then vanishing cycle is X1-X0}. Let $Z=\{df=0 \}$ be the
degeneracy locus of $f$ and let $Z_{\aff }\subset X_{\aff }$ be the
affinizations of $Z$ and $X$. Suppose that $X_{0}=f^{-1} (0)$ is
reduced. Then $[\phi_f]_{Z_{\aff }}$, the motivic
vanishing cycle of~$f$ relative to $Z_{\aff }$, lies in the subring
$\M_{Z_{\aff }}\subset\M_{Z_{\aff }}^\muhat$.
\end{proposition}

This result will also be proved in Appendix~\ref{appendix: nearby
fiber proof}.

\subsection{The virtual motive of a degeneracy locus}

\begin{definition}\label{defn: virtual motive}
Let $f\colon X\to\C$ be a regular function on a smooth variety~$X$,
and let
\[Z=\{df=0\}\subset X\]
be its degeneracy locus. We define the \emph{relative virtual
motive} of~$Z$ to be
\[ 
[Z]_\relvir = -\LL ^{-\frac{\dim X}{2}}[\phi_f]_Z\in\M^\muhat_Z,
\]
and the \emph{absolute virtual motive} of $Z$ to be 
\[ [Z]_\vir = -\LL ^{-\frac{\dim X}{2}}[\phi_f]\in\M^\muhat_\C,
\]
the pushforward of the relative virtual motive $[Z]_\relvir$ 
to the absolute motivic ring~$\M^\muhat_\C$.
\end{definition}

\begin{remark}\label{rem: the case when Z= (df=0) is smooth}
As a degenerate but important example, consider $f=0$. Then we have
$X_0=X$ and $[\psi_f]_{X_0}=0$, so the virtual motives of a smooth
variety $X$ with $f=0$ are given by
\begin{equation}\label{eqn: relvir for smooth X}
[X]_\relvir = \LL ^{-\frac{\dim X}{2}}[X]_X\in \M _{X}\subset \M
^{\muhat }_{X}
\end{equation}
and
\begin{equation}\label{eqn: vir motive for smooth X}
[X]_\vir = \LL ^{-\frac{\dim X}{2}}[X] \in \M _{\C }\subset \M
^{\muhat }_{\C }.
\end{equation}
If one imposes~\eqref{eqn: sqrt (L) = 1-[2pt,mu2])} as a
relation in $\M ^{\muhat }_{\C }$, then it is not hard to show that
whenever $Z$ is smooth, $[Z]_{\relvir }$ agrees with
equation~\eqref{eqn: relvir for smooth X} at each point, that
is for each $P\in Z$, $[Z]_{\relvir }|_{P} = \LL ^{-\frac{\dim
Z}{2}}$.
\end{remark}

\begin{proposition} 
$\,$
\begin{enumerate}
\item At a point $P\in Z$, the fiberwise Euler characteristic of the
relative virtual motive $[Z]_\relvir\in\M^\muhat_Z$, evaluated
at the specialization $\LL ^\half=-1$, is equal to the value at
$P\in Z$ of the constructible function
$\nu_Z$ of~\cite{Behrend-micro}.
\item The Euler characteristic of the absolute virtual motive 
$[Z]_\vir\in\M_\C^\muhat$ is the virtual Euler characteristic 
$\chi_\vir(Z)\in\Z$ of~\cite{Behrend-micro}:
\[
\chi ([Z]_{\vir }) = \chi _{\vir } (X) = \sum _{k\in \Z } k\,\chi \left(\nu ^{-1}_{Z} (k)
\right).
\]
\end{enumerate}
\label{prop_spec_to_kai}
\end{proposition}
\proof By~\cite[Eq~(4)]{Behrend-micro}, for $Z=\{df=0\}\subset X$, the
value of the function~$\nu_Z$ at the point~$P$ is
\[ \nu_Z(P)=(-1)^{\dim X}(1-\chi(F_P)),
\]
where $F_P$ is the Milnor fiber of~$f$ at~$P$. On the other hand, the
pointwise Euler characteristic of $[\phi_f]$ at $P$ is the Euler
characteristic of the reduced cohomology of the Milnor fiber $F_P$
\cite[Thm.~3.5.5]{Denef-Loeser-geometry}, equal to
$\chi(F_P)-1$. The factor $-\LL ^{\dim X/2}$ at $\LL ^\half=-1$
contributes the factor $-(-1)^{\dim X}$. This proves~(1).  (2) clearly
follows from~(1).  \endproof

\begin{remark}
If $Z=\{df=0 \}$ is a moduli space of sheaves on a \CY
threefold, then the associated \DT invariant is given by $\chi _{\vir
} (Z)$. So by the above proposition, $[Z]_{\vir }$ is a motivic
refinement of the \DT invariant and hence can be regarded as a
\emph{motivic \DT invariant}. The function $f$ in this context is
called a \emph{global Chern-Simons functional} or a
\emph{super-potential.}
\end{remark}

\begin{remark} Unlike the ordinary motivic class of $Z$, the virtual
motive is sensitive to both the singularities and the scheme structure
of~$Z$ since in particular, the constructible function $\nu _{Z}$
is. However, unlike the function $\nu _{Z}$, we expect the virtual
motive of $Z$ to depend on its presentation as a degeneracy locus
$Z=\{df=0 \}$ and not just its scheme structure. We will not include
the pair $(X,f)$ in the notation but it will be assumed that whenever
we write $[Z]_\vir$, it is to be understood with a particular choice
of~$(X,f)$. When $Z$ is smooth, a canonical choice is provided by
$(X,f)= (Z,0)$.
\end{remark}

\begin{remark} Let us comment on our use of the term {\em virtual},
which has acquired two different meanings in closely related
subjects. On the one hand, there are what are sometimes called {\em
virtual invariants}, invariants of spaces which are additive under the
scissor relation; these invariants are generally called {\em motivic}
in this paper. Examples include the {\em virtual Hodge polynomial} and
the {\em virtual Poincar\'e polynomial}, for which we use the terms
E-polynomial and weight polynomial. On the other hand, there is the
philosophy of {\em virtual smoothness} and the technology of {\em
virtual cycles} for spaces such as Hilbert schemes of a threefold,
which have excess dimension compared to what one would expect from
deformation-obstruction theory. We use the term {\em virtual}
exclusively in this second sense in this paper.
\end{remark}

\section{The Hilbert scheme of points on $\C^3$}

\subsection{Generalities on Hilbert schemes of points}

For a smooth and quasi-projective variety $X$ of dimension~$d$, let 
\[\Sym^n(X)=X^n/S_n\] denote
the $n$-th symmetric product of $X$. For a partition $\alpha$ of $n$, 
let \[\Sym^n_\alpha(X)\subset \Sym^n(X)\]
denote the locally closed subset of $\Sym^n(X)$ 
of $n$-tuples of points with multiplicities 
given by~$\alpha$. This gives a stratification
\[ \Sym^n(X) = \coprod_{\alpha\vdash n} \Sym_\alpha^n(X).
\] 
Assume that the number of parts in~$\alpha$ is~$n(\alpha)$ and that $\alpha$ 
contains~$\alpha_i$ parts of length~$i$. Let 
$G_\alpha =\prod_i S_{\alpha_i}$ be the corresponding 
product of symmetric groups. Then 
there exists an open set $T_\alpha\subset X^{n(\alpha)}$ such that
\[ \Sym_\alpha^n(X) = T_\alpha/G_\alpha
\]
is a free quotient. 

There is a similar story for the Hilbert scheme. The Hilbert scheme 
$\Hilb^n(X)$ is stratified 
\[ \Hilb^n(X) = \coprod_{\alpha\vdash n} \Hilb_\alpha^n(X)
\] 
into locally closed strata $\Hilb_\alpha^n(X)$, the preimages of 
$\Sym_\alpha^n(X)$ under the Hilbert--Chow morphism. On the deepest stratum 
with only one part, 
\[\Hilb^n_{(n)}(X)\to \Sym^n_{(n)}(X) \cong X\]
is a known to be a Zariski locally trivial fibration with 
fiber $\Hilb^n(\C^d)_0$, the punctual Hilbert scheme of affine $d$-space 
at the origin; see e.g.~\cite[Corollary 4.9]{Behrend-Fantechi08}. For affine 
space, we have a product
\[\Hilb^n_{(n)}(\C^d)\cong \C^d\times\Hilb^n(\C^d)_0.\]
For an arbitrary partition $\alpha$, by e.g.~\cite[Lemma 4.10]{Behrend-Fantechi08},
\[\Hilb^n_\alpha(X) = V_\alpha/G_\alpha
\]
is a free quotient, with
\[
V_\alpha = \prod_i \left(\Hilb_{(i)}^{i}(X)\right)^{\alpha_i}\setminus\tilde\Delta,
\]
where $\tilde\Delta$ denotes the locus of clusters with intersecting support. The product of
Hilbert--Chow morphisms gives a map
\[ V_\alpha\to \prod_iX^{\alpha_i}\setminus \Delta,\]
where as before, $\Delta$ is the big diagonal in a product of copies of~$X$. 
This map is a Zariski locally trivial fibration with fiber 
$\prod_i(\Hilb^i(\C^d)_0)^{\alpha_i}$.

\subsection{The Hilbert scheme of $\C^3$ as critical locus} \label{subsec: the hilb scheme as a crit locus}

Let $T$ be the three-dimensional space of linear functions on $\C^3$,
so that \[\C^3=\Spec\Sym^\bullet T.\] Fix an isomorphism $\vol: \wedge^3
T\cong \C$; this corresponds to choosing a holomorphic volume form
(\CY form) on $\C^3$. We start by recalling the
de\-scrip\-tion of the Hilbert scheme as a degeneracy locus from
\cite[Proof of Thm.~1.3.1]{Szendroi-GT}.

\begin{proposition} The pair $(T,\vol)$ defines an embedding of the Hilbert 
scheme $\Hilb^n(\C^3)$ into a smooth quasi-projective variety $M_n$, which in 
turn is equipped with a regular function $f_n\colon M_n\to\C$, such that
\begin{equation} \Hilb^n(\C^3) = \{ df_n=0\}\subset M_n\label{eq_emb}\end{equation}
is the scheme-theoretic degeneracy locus of the function $f_n$ on $M_n$.
\label{prop:sup}
\end{proposition}

\begin{proof} A point $[Z]\in \Hilb^n(\C^3)$ corresponds to an
embedded $0$-dimen\-sio\-nal subscheme $Z\hookrightarrow\C^3$ of length
$n$, in other words to a quotient $\O_{\C^3}\to \O_Z$ with $H^0(\O_Z)$
of dimension $n$.  Fixing an $n$-dimensional complex vector
space~$V_n$, the data defining a cluster consists of a linear map
$T\otimes V_n\to V_n$, subject to the condition that the induced
action of the tensor algebra of $T$ factors through an action of the
symmetric algebra $\Sym^\bullet T$, and a vector $1\in V_n$ which
generates $V_n$ under the action.

Let
\[ U_n\subset \Hom(T\otimes V_n, V_n) \times V_n\]
denote the space of maps with cyclic vector, the open subset where the
linear span of all vectors obtained by repeated applications of the
endomorphisms to the chosen vector $v\in V_n$ is the whole $V_n$. 

Let $\chi :\GL (V_{n})\to \C ^{*}$ be the character given by $\chi
(g)=\det (g)$.  As proved in~\cite[Lemma 1.2.1]{Szendroi-GT}, the open
subset~$U_n$ is precisely the subset of stable points for the action
of $\GL (V_{n})$ linearized by $\chi$. In particular, the action of
$\GL(V_n)$ on $U_n$~is free, and the quotient
\[ M_n = \Hom(T\otimes V_n, V_n) \times V_n/\!/_\chi \GL(V_n)= U_n/\GL(V_n)\]
is a smooth quasi-projective GIT quotient. 

Finally consider the map 
\[ \phi \mapsto  \Tr\left(\wedge^3\phi\right),\]
where $\wedge^3\phi:\bigwedge^3 T\times V_n\to V_n$ and we use the
isomorphism $\vol$ before taking the trace on~$V_n$.  It is clear that
this map descends to a regular map $f_n\colon M_n\to\C$.  The
equations $\{df_n=0\}$ are just the equations which say that the
action factors through the symmetric algebra; this is easy to see from
the explicit description of Remark~\ref{rem_expl} below. Finally, as
proved by \cite{Nakajima-Hilbert} (in dimension 2, but the proof
generalizes), the scheme cut out by these equations is precisely the
moduli scheme representing the functor of~$n$ points on~$\C^3$. Thus,
as a scheme,
\[ \Hilb^n(\C^3)=\{df_n=0\} \subset M_n.
\]
\end{proof}

\begin{remark} \rm Fixing a basis of~$V_n$, the commutative algebra
$\C[x,y,z]$ acts on $V_n$ by a triple of matrices $A,B,C$. The variety
$M_n$ is the space of triples with generating vector, where the
matrices do not necessarily commute, modulo the action of $GL
(V_{n})$. The map $f_n$ on triples of matrices is given by
\[ (A,B,C)\mapsto {\rm Tr} [A,B]C.\]
Written explicitly in terms of the matrix entries, 
\[  {\rm Tr} [A,B]C = \sum_{i,k} \sum_j \left(A_{ij} B_{jk}- B_{ij} A_{jk} \right) C_{ki},\]
and so
\[ \partial_{C_{ki}} {\rm Tr} [A,B]C = \sum_j \left(A_{ij} B_{jk}-B_{ij} A_{jk}\right) = 0
\]
for all $i,k$ indeed means that $A$ and $B$ commute. 
\label{rem_expl}\end{remark}

\begin{remark}
This description of $\Hilb ^{n} (\C ^{3})$ can also be written in the
language of quivers. Consider the quiver consisting of two nodes, with a
single arrow from the first node to the second (corresponding to $v$) and
three additional loops on the second node (corresponding to $A$, $B$, and
$C$), with relations coming from the super-potential $W=A[B,C]$. 
Then the space $M_{n}$ can be identified with stable representations 
of this bound quiver, with dimension vector $(1,n)$, and specific 
choice of stability parameter which matches the GIT stability condition
described above. 
\end{remark}

\subsection{Some properties of the family}

\begin{lemma} There exists a linearized $T=(\C^*)^{3}$-action on $M_n$
such that $f_n\colon M_n \to \C$ is equivariant with respect to the
primitive character $\chi :T\to \C ^{*}$ given by $\chi
(t_{1},t_{2},t_{3})=t_{1}t_{2}t_{3}$. Moreover, the action of the
diagonal 1-parameter subgroup is circle compact.
\end{lemma}
\begin{proof} In the notation of Remark~\ref{rem_expl}, 
consider the $T$-action
\[ 
(t_{1},t_{2},t_{3})\circ (A,B,C,v) =(t_{1}
A,t_{2}B,t_{3}C,t_{1}t_{2}t_{3}v)
\]
on the space 
\[U_n\subset\Hom(V_n,V_n)^{\times 3}\times V_n\] 
of maps with cyclic vector.  The map $(A,B,C,v)\mapsto {\rm Tr}
[A,B]C$ is $T$-equivariant with respect to the character $\chi
(t_{1},t_{2},t_{3})=t_{1}t_{2}t_{3}$. Moreover, the $T$-action on
$U_n$ commutes with the $\GL(V_n)$-action acting freely on~$U_n$, so
descends to the quotient $M_n$. The $T$-action on~$U_{n}$ lifts to the
linearization and hence defines a linearized action of~$T$ on~$M_{n}$.

Consider $\C^*$-action on $\Hom(V_n,V_n)^{\times 3}\times V_n$ induced
by the diagonal subgroup in $T$.  Let \[M^0_n=\Hom(V_n,V_n)^{\times
3}\times V_n /\!/_0 \GL(V_n)\] be the affine quotient, the GIT
quotient at zero stability. Then by general GIT, there is a natural
proper map $\pi_n\colon M_n\to M^0_n$, which is $\C ^{*}$-equivariant.
On the other hand, it is clear that the only $\C^*$-fixed point in
$M^0_n$ is the image of the origin in $\Hom(V_n,V_n)^{\times 3}\times
V_n$, and all $\C^*$-orbits in $M^0_n$ have this point in their
closure as $\lambda\to 0$.  By the properness of $\pi_n$, the
$\C^*$-fixed points in $M_n$ form a complete subvariety, and all
limits as $\lambda\to 0$ exist. Thus the diagonal $\C^*$-action on
$M_n$ is circle compact.
\end{proof}

The function $f_n:M_{n}\to \C $ is not proper, so we cannot apply
Proposition~\ref{prop: if f is proper and C* equivariant then
vanishing cycle is X1-X0}, but as a corollary to the above lemma, we
may apply Proposition~\ref{prop: if f is C* equiv and circle compact,
then vanishing cycle is X1-X0} instead.

\begin{corollary}\label{cor: motivic_difference}
For each $n$, the absolute motivic vanishing cycle of the
family~$f_n\colon M_n \to \C$ can be computed as the motivic
difference
\[  
[\phi_{f_n}]=[f_n^{-1}(1)]-[f_n^{-1}(0)]\in\M_\C\subset\M_\C^\muhat.
\]
\end{corollary}

\subsection{The virtual motive of the Hilbert scheme} 
\label{subsec: virtual motives for C3}  As a consequence
of Proposition~\ref{prop:sup}, the singular space $\Hilb^n(\C^3)$
acquires relative and absolute virtual motives
$[\Hilb^n(\C^3)]_\relvir$ and $[\Hilb^n(\C^3)]_\vir$.  Let us stress
that, a priori, these classes depend on the chosen linear \CY
structure on~$\C^3$.

Define the motivic \DT  partition function of $\C^3$ to be
\[
Z_{\C^3}(t) = \sum_{n=0}^\infty [\Hilb^n(\C^3)]_\vir \,t^n
\in\M_\C[[t]].
\] 
By Proposition~\ref{prop_spec_to_kai},
\[ 
W\left([\Hilb^n(\C^3)]_\vir, q^\half=-1\right)\in\Z
\]
is the \DT invariant, the signed number of 3-dimen\-sio\-nal partitions
of~$n$. Hence the Euler characteristic specialization
\[ 
\chi Z_{\C^3}(t)= M(-t)
\]
is the signed MacMahon function.

Using the stratification
\[ \Hilb^n(\C^3) = \coprod_{\alpha\vdash n} \Hilb_\alpha^n(\C^3),
\]
the relative virtual motive
\[ [\Hilb^n(\C^3)]_\relvir\in\M^\muhat_{\Hilb^n(\C^3)}
\]
can be restricted to define the relative virtual motives 
\[[\Hilb_\alpha^n(\C^3)]_\relvir\in \M^\muhat_{\Hilb_\alpha^n(\C^3)}\]
for the strata. We additionally define the relative virtual motive of
the punctual Hilbert scheme
\[
[\Hilb ^{n} (\C ^{3})_{0}]_{\relvir }\in \M ^{\muhat }_{\Hilb ^{n} (\C
^{3})_{0}}
\]
by restricting $[\Hilb ^{n}_{(n)} (\C ^{3})]_{\relvir }$ to
\[
\{0 \}\times \Hilb ^{n} (\C ^{3})_{0}\subset \C ^{3}\times \Hilb ^{n}
(\C ^{3})_{0}\cong \Hilb ^{n}_{(n)} (\C ^{3}).
\]
Associated to each relative virtual motive, we have the absolute motives
\[
[\Hilb_\alpha^n(\C^3)]_\vir\,,\,\,[\Hilb ^{n} (\C^{3})_{0}]_\vir\in\M^\muhat_\C,
\] 
and the absolute motives satisfy 
\[ 
[\Hilb^n(\C^3)]_\vir = \sum_{\alpha\vdash n}
[\Hilb_\alpha^n(\C^3)]_\vir.
\]
We now collect some properties of these virtual motives. 
\begin{proposition}\begin{enumerate}
\item The absolute virtual motives $[\Hilb_\alpha^n(\C^3)]_\vir$
and $[\Hilb^n(\C^3)_0]_\vir$ live in the subring $\M_\C\subset\M_\C^\muhat$.
\item On the closed stratum, 
\[[\Hilb^n_{(n)}(\C^3)]_\vir = \LL^3\cdot[\Hilb^n(\C^3)_0]_\vir\in\MC.\] 
\item More generally, for a general stratum, 
\[ [\Hilb^n_\alpha(\C^3)]_\vir = \pi_{G_\alpha}\left(\left[\prod_i (\C^3)^{\alpha_i}\setminus \Delta\right]\cdot \prod_i\left[\Hilb^{i}(\C^3)_0^{\alpha_i}\right]_\vir\right),\]
where $\pi_{G_\alpha}$ denotes the quotient map~\eqref{eq: quotient map on rings}.
\end{enumerate}
\label{prop_C3properties}\end{proposition} 

\proof We start by proving (1) and (2) together. On the one hand, consider the closed stratum 
\[\Hilb^n_{(n)}(\C^3)\cong \C^3\times\Hilb^n(\C^3)_0,\] 
with projections $p_i$ to the factors.
By the invariance of the construction under the translation
action of $\C^3$ on itself, the relative virtual motive is
\[ [\Hilb^n_{(n)}(\C^3)]_\relvir= p_2^*[\Hilb^n(\C^3)_0]_\relvir.
\]
Taking absolute motives, 
\begin{equation} [\Hilb^n_{(n)}(\C^3)]_\vir = \LL^3\cdot[\Hilb^n(\C^3)_0]_\vir,
\label{eq deepest stratum}\end{equation}
with both sides living a priori in ~$\M_\C^\muhat$.

On the other hand, as it is well known, the affinization of the Hilbert scheme $\Hilb^n(\C^3)$ is 
the symmetric product. The conditions of
Proposition~\ref{prop: relative and effective version of vanishing cycle = X1 - X0 proposition}, 
hold, since the cubic hypersurface given by the function $f_n$ is reduced.
Applying Proposition~\ref{prop: relative and effective version of vanishing cycle = X1 - X0 proposition}, 
we see that the relative virtual motives on the strata of the Hilbert scheme have 
trivial $\muhat$-action over the corresponding strata in the symmetric product. Hence
the absolute motives $[\Hilb^n_{\alpha}(\C^3)]_\vir$ also carry trivial $\muhat$-action. 
The same statement for the punctual Hilbert scheme then follows from~\eqref{eq deepest stratum},
with~\eqref{eq deepest stratum} holding in fact in $\MC$. 

To prove (3), consider the diagram
\[\begin{diagram}[height=1.8em,width=1.8em,nohug]
V_\alpha && \rInto && W_\alpha && \rInto && \displaystyle\prod_i  \Hilb^{i}(\C^3)^{\alpha_i}\\
\dTo &&  && \dTo \\ 
\\
\Hilb^n_\alpha(\C^3) && \rInto && U_\alpha && \rInto && \Hilb^n(\C^3)\\
&&&&
\end{diagram}\] from~\cite[Lemma 4.10]{Behrend-Fantechi08}. Here
$W_\alpha$ is the locus of points in the product $\prod_i
\Hilb^{i}(\C^3)^{\alpha_i}$ which parametrizes
subschemes with disjoint support. The first vertical map is Galois
whereas the second one is \'etale. The first inclusion in each row is
closed whereas the second one is open. 

Consider the construction of the Hilbert scheme, as a space of commuting matrices with 
cyclic vector, in a neighborhood of~$U_\alpha$ in the space of
matrices. Pulling back to the cover~$V_\alpha$, we see that a point of
$V_\alpha$ is represented by tuples of commuting matrices $X_j, Y_j,
Z_j$ acting on some linear spaces $V_j$, with generating vectors
$v_j$. The covering map is simply obtained by direct sum: $V=\oplus
V_j$ acted on by $X=\oplus X_j$ and $Y$, $Z$ defined similarly.  The
vector $v=\oplus v_j$ is cyclic for $X,Y,Z$ exactly because the
eigenvalues of the $X_j, Y_j, Z_j$ do not all coincide for
different~$j$; this is the disjoint support property of points
of~$V_\alpha$.  

On the other hand, clearly
\[ \Tr X[Y,Z] = \sum_j \Tr X_j[Y_j, Z_j]
\]
for block-diagonal matrices. Thus, the Thom--Sebastiani Theorem~\ref{thm_thom_seb}
implies that the pullback relative motive 
\[q_\alpha^*[\Hilb_\alpha^n(\C^3)]_\relvir\in \M^\muhat_{V_\alpha}\]
is equal to the restriction to $V_\alpha$ of the $\star$-products of
the relative virtual motives of the punctual Hilbert schemes
$\Hilb_{(i)}^{i}(\C^3)$. Taking the associated absolute motives, using the locally trivial
fibration on $V_\alpha$ along with (2), we get
\[ q_\alpha^*[\Hilb_\alpha^n(\C^3)]_\vir=\left[\prod_i (\C^3)^{\alpha_i}\setminus \Delta\right]\cdot \prod_i\left[\Hilb^{i}(\C^3)_0^{\alpha_i}\right]_\vir\in\MC.
\] 
Here, using (1), the $\star$-product became the ordinary product.
By Lemma \ref{lemma: symm grp acts}, the $G_\alpha$-action extends to this
class, and (3) follows.
\endproof

\subsection{Relationship to the Kontsevich--Soibelman definition}

Let $E$ be an object in an ind-constructible \CY
$A_\infty$-category~$\CC$ (see~\cite{Kontsevich-Soibelman} for the
definitions of all these terms). Kontsevich and Soibelman associate to
$E$ a motivic weight $w(E)$. This weight lives in a certain motivic
ring $\bar \M^\muhat_\C$, which is a completion of the ring
$\M_\C^\muhat$ used above, quotiented by the equivalence relation,
explained in~\cite[Section 4.5]{Kontsevich-Soibelman}, which
essentially says that two motivic classes are equivalent if all
cohomological realizations of these classes coincide. They claim
moreover that, given a moduli space $S$ of objects of $\CC$, their
definition gives an element in a piecewise-relative motivic ring
$\bar\M_S^\muhat$.

\cite[Definition 17]{Kontsevich-Soibelman} appears closely related to
our definition of the virtual motive of a degeneracy locus above. It
relies on a local description of the moduli space $S$ as the zeros of
a formal functional $W$ cooked up from the $A_\infty$-structure on
$\CC$.  The definition is essentially like ours, using the motivic
vanishing cycle of the (local) function $W$, twisted by half the
dimension of $S$, as well as certain additional factors arising from a
choice of what they call orientation data on $\CC$.

The relevant category for our discussion is an $A_\infty$-enhancement
of some framed version of the derived category of sheaves on $\C^3$,
more precisely the subcategory thereof generated by the structure
sheaf of $\C^3$ and structure sheaves of points. Ideal sheaves of
points in this category have as their moduli space the Hilbert
scheme. It would be interesting to construct directly some
$A_\infty$-structure on this category, as well as a consistent set of
orientation data (compare~\cite{Davison}). 
Should this be possible, we expect that our motivic
invariants, perhaps up to some universal constants, agree with the
Kontsevich--Soibelman definition; compare \cite[end of Section
7.1]{Kontsevich-Soibelman}. For the numerical invariants of the
Hilbert scheme, see also the discussion in \cite[Section
6.5]{Kontsevich-Soibelman}.

\subsection{Computing the motivic partition function of $\C^3$} 
$\,$

\smallskip

The core result of this paper is the computation of 
\[
Z_{\C ^{3}} (t) = \sum _{n=0}^{\infty } [\Hilb ^{n} (\C ^{3})]_{\vir
}\, t^{n},
\]
the motivic \DT partition function of $\C ^{3}$.

\begin{theorem} The motivic partition function $Z_{\C ^{3}} (t)$ lies
in $\M _{\C }[[t]] \subset \M^{\muhat } _{\C }[[t]] $ and is given by
\begin{equation}
Z_{\C^3}(t) = \prod_{m=1}^{\infty } \prod _{k=0}^{m-1}
\left(1-\LL ^{k+2-m/2}t^{m} \right)^{-1}.
\label{eq_hilb_C3}
\end{equation}
\label{thm_C3}
\end{theorem}
\proof Recall that \[\Hilb^n(\C^{3})=\{df_n=0 \},\] where $f_n$~is the
function
\[
f_n (A,B,C,v) = \Tr A[B,C]
\] 
defined on the smooth variety
\[
M_n=U_n/\GL(V_{n}),
\]
where $V_n$ is an $n$-dimensional vector space, and 
\[ U_n\subset \Hom (V_{n},V_{n})^{3}\times V_{n}
\]
is the open set of points $(A,B,C,v)$ satisfying the stability condition that
monomials in $A,B,C$ applied to $v$ generate $V_{n}$.

By Corollary~\ref{cor: motivic_difference}, to compute the virtual
motive, we need to compute the motivic difference of the fibers
$f_n^{-1}(1)$ and $f_n^{-1}(0)$.

Let 
\[
Y_{n} = \{(A,B,C,v): \Tr A[B,C]=0 \}\subset \Hom (V_{n},V_{n})^{3}\times
V_{n},
\]
and let
\[
Z_{n} = \{(A,B,C,v): \Tr A[B,C]=1 \}\subset \Hom (V_{n},V_{n})^{3}\times
V_{n}.
\]
The isomorphism
\[
\Hom (V_{n},V_{n})^{3} \times V_{n}\backslash Y_{n} \cong \C ^{*}\times Z_{n}
\] 
given by 
\[
(A,B,C,v)\mapsto \left(\Tr A[B,C], (\Tr A[B,C])^{-1}A,B,C \right)
\]
yields the motivic relation
\[
[Y_{n}] + (\LL -1) [Z_{n}] = [\Hom (V_{n},V_{n})\times
V_{n}] = \LL ^{3n^{2}+n},
\]
equivalently
\[
(1-\LL )\left([Y_{n}]-[Z_{n}] \right) = \LL ^{3n^{2}+n} - \LL  [Y_{n}].
\]

The space $Y_{n}$ stratifies as a union
\[
Y_{n}=Y_{n}' \sqcup Y_{n}''
\]
where $Y_{n}'$ consists of the locus where $B$ and $C$ commute and
$Y_{n}''$ is its complement. Projections onto the $B$ and $C$ factors
induce maps
\[
Y_{n}'\to C_{n},\quad Y_{n}''\to \{\C^{2n^{2}}  \setminus C_{n} \}
\]
where $C_{n}\subset \C^{2n^{2}}$ is the commuting variety. The first
map splits as a product $Y_{n}'\cong \C^{n^{2}+n}\times C_{n}$ and
the second map is a Zariski trivial fibration with fibers isomorphic
to $\C^{n^{2}-1+n}$. Indeed, for fixed $B$ and $C$ with $[B,C]\neq
0$, the condition $\Tr A[B,C]=0$ is a single non-trivial linear
condition on the matrices $A$. Moreover, the fibration is Zariski
trivial over the open cover whose sets are given by the condition that
some given matrix entry of $[B,C]$ is non-zero.
Thus the above stratification yields the equation of motives
\[
[Y_{n}] = \LL ^{n^{2}+n}[C_{n}] +\LL ^{n^{2}-1+n}\left(\LL ^{2n^{2}}-[C_{n}]
\right).
\]

Substituting into the previous equation and canceling terms we obtain
\begin{align*}
(1-\LL )\left([Y_{n}]-[Z_{n}] \right) = -\LL ^{n^{2}+n} (\LL [C_{n}]-[C_{n}]).
\end{align*}
Writing 
\[w_{n} = [Y_{n}]-[Z_{n}],\]
we get the basic equation
\begin{equation}\label{eqn: Wn=L^n(n+1)Cn}
w_{n} = \LL ^{n (n+1)}[C_{n}].
\end{equation}

We now need to incorporate the stability condition. We call the
smallest subspace of $V_{n}$ containing $v$ and invariant under the
action of $A$, $B$, and $C$ the \emph{$(A,B,C)$-span of $v$}. Let
\[
X^{k}_{n} = \{(A,B,C,v):\text{the $(A,B,C)$-span of $v$ has dimension
$k$} \}
\]
and let
\begin{align*}
Y_{n}^{k}& = Y_{n}\cap X^{k}_{n},\\
Z_{n}^{k}& =Z_{n}\cap X_{n}^{k}.
\end{align*}

We compute the motive of $Y_{n}^{k}$ as follows. There is a Zariski
locally trivial fibration
\[
Y^{k}_{n}\to Gr (k,n)
\]
given by sending $(A,B,C,v)$ to the $(A,B,C)$-span of $v$.

To compute the motive of the fiber of this map, we choose a basis of
$V_{n}$ so that the first $k$ vectors are in the $(A,B,C)$-span of
$v$. In this basis, $(A,B,C,v)$ in a fixed fiber all have the form
\[
A=\left(\begin{matrix} A_{0}&A'\\0&A_{1}  \end{matrix} \right)\quad 
B=\left(\begin{matrix} B_{0}&B'\\0&B_{1}  \end{matrix} \right)\quad 
C=\left(\begin{matrix} C_{0}&C'\\0&C_{1}  \end{matrix} \right)\quad 
v=\left(\begin{matrix} v_{0}\\0  \end{matrix} \right)\quad 
\]
where $(A_{0},B_{0},C_{0})$ are $k\times k$ matrices, $(A',B',C')$ are
$k\times (n-k)$ matrices, $(A_{1},B_{1},C_{1})$ are $(n-k)\times
(n-k)$ matrices, and $v_{0}$ is a $k$-vector. 

Thus a fiber of $Y^{k}_{n}\to Gr (k,n)$  is given by the locus of
\[
\{(A_{0},B_{0},C_{0},v_{0}),(A_{1},B_{1},C_{1}),(A',B',C')\}
\]
satisfying
\[
\Tr A[B,C] = \Tr A_{0}[B_{0},C_{0}]+\Tr A_{1}[B_{1},C_{1}]=0.
\]
This space splits into a factor $\C^{3 (n-k)k}$, corresponding to the
triple $(A',B',C')$, and a remaining factor which stratifies into a union of 
\[
\left\{\Tr A_{0}[B_{0},C_{0}]=\Tr A_{1}[B_{1},C_{1}]=0 \right\}
\]
and
\[
\left\{\Tr A_{0}[B_{0},C_{0}]=-\Tr A_{1}[B_{1},C_{1}]\neq 0 \right\}.
\]
Projection on the $(A_{0},B_{0},C_{0},v_{0})$ and
$(A_{1},B_{1},C_{1})$ factors induces a product structure on the above
strata so that the corresponding motives are given by
\[
[Y^{k}_{k}]\cdot [Y_{n-k}]\LL ^{- (n-k)}
\]
and
\[
(\LL -1)[Z^{k}_{k}][Z_{n-k}]\LL ^{- (n-k)}
\]
respectively. Putting this all together yields
\begin{eqnarray*}
[Y_{n}^{k}]& = & \LL ^{3 (n-k)k}\Lbinom{n}{k}\left([Y^{k}_{k}]\cdot [Y_{n-k}]\cdot
\LL ^{- (n-k)}  \right.\\
& & \ \ \ \ \ \ \ \ \ \ \ \ \ \ \ \ \ \ +\left.(\LL -1)\cdot [Z^{k}_{k}]\cdot [Z_{n-k}]\cdot \LL ^{- (n-k)}
\right).
\end{eqnarray*}
A similar analysis yields
\begin{multline*}
[Z^{k}_{n}]  =  \LL ^{3 (n-k)k}\Lbinom{n}{k}\left([Y^{k}_{k}]\cdot [Z_{n-k}]\cdot
\LL ^{- (n-k)} \right. +\\
 (\LL -2) \cdot[Z^{k}_{k}]\cdot [Z_{n-k}]\cdot \LL ^{- (n-k)}  +\left.[Z^{k}_{k}]\cdot [Y_{n-k}]\cdot \LL ^{- (n-k)} \right).
\end{multline*}
We are interested in the difference
\begin{align*}
w^{k}_{n}&=[Y_{n}^{k}]-[Z^{k}_{n}]\\
&=\LL ^{(3k-1) (n-k)}\Lbinom{n}{k}\left(w_{n-k}[Y^{k}_{k}] -
w_{n-k}[Z^{k}_{k}] \right)\\
&=\LL ^{(n-k) (n+2k)}\Lbinom{n}{k}\,[C_{n-k}]\,w^{k}_{k},
\end{align*}
where we used~\eqref{eqn: Wn=L^n(n+1)Cn} for the last equality.

Observing that $Y_{n}=\sqcup_{k=0}^{n}Y_{n}^{k}$ and $Z_{n}=\sqcup
_{k=0}^{n}Z^{k}_{n}$, we get
\[
w^{n}_{n} = w_{n} - \sum _{k=0}^{n-1}w^{k}_{n},
\] 
into which we substitute our equations for $w_{n}$ and $w^{k}_{n}$ to
derive the following recursion for $w^{n}_{n}$:
\begin{equation}\label{eqn: recursion for Wnn}
w^{n}_{n}=\LL ^{n (n+1)}[C_{n}] -\sum _{k=0}^{n-1} \Lbinom{n}{k} \LL ^{(n-k)
(n+2k)}[C_{n-k}]w^{k}_{k}.
\end{equation}

We can now compute the virtual motive of the Hilbert scheme. By
Proposition~\ref{prop: if f is C* equiv and circle compact, then
vanishing cycle is X1-X0}, we get
\begin{align*}
[\phi_{f_n}] &= - [f_n^{-1}(0)]+[f_n^{-1}(1)] \\
&=- \frac{[Y_{n}^{n}]}{[\GL _{n} (\C )]} +\frac{[X^{n}_{n}]}{[\GL _{n} (\C )]}\\
&= -\frac{{w}^{n}_{n}}{\LL ^{\binom{n}{2}}\Lfact{n}}.
\end{align*}
The dimension of $M_n$ is $2n^{2}+n$, so we find
\begin{align*}
[\Hilb^n(\C^{3})]_{\vir} =& -\LL ^{-n^{2}-n/2}[\phi _{f_{n}}]\\
=&\LL  ^{-\frac{3n^{2}}{2}} \frac{w^{n}_{n}}{\Lfact{n}}.
\end{align*}

Working in the ring  $\M_\C[(1-\LL ^n)^{-1}\colon {n\geq 1}]$, 
we divide~\eqref{eqn: recursion for Wnn} by
$\LL ^{3n^{2}/2}\Lfact{n}$ and rearrange to obtain
\[
\tilde{c}_{n}\, \LL ^{n/2} = \sum _{k=0}^{n} \tilde{c}_{n-k}\,[\Hilb ^{k} (\C ^{3})]_{\vir }\, \LL ^{-(n-k)/2},
\]
where 
\[
\tilde{c}_{n} = \LL ^{-\binom{n}{2}}\frac{C_{n}}{\Lfact{n}}
\]
is the renormalized motive~\eqref{eq:tildecn} of
the space~$C_n$ of commuting pairs of matrices. Multiplying by $t^{n}$
and summing, we get
\[
C(t\LL ^{1/2}) = Z_{\C ^{3}}(t) C(t\LL ^{-1/2}),
\]
with $C(t)$ as in~\eqref{eq_Cseries}. Thus using
Proposition~\ref{prop: Feit-Fine} (cf. Remark~\ref{rem: coeffs of the
feit-fine formula are rational fncs in L}) we obtain
\begin{align}
Z_{\C ^{3}}(t)&= \frac{C (t\LL ^{1/2})}{C (t\LL ^{-1/2})}\label{eq:twisted}\\
&=\prod _{m=1}^{\infty }\prod _{j=0}^{\infty }\frac{\left(1-\LL
^{1-j+m/2
}t^{m} \right)^{-1}}{\left(1-\LL ^{1-j-m/2 }t^{m} \right)^{-1}}\nonumber \\
&=\prod _{m=1}^{\infty }\prod _{j=0}^{m-1} \left(1-\LL ^{1-j+m/2}t^{m}
\right)^{-1}\nonumber\\
&=\prod _{m=1}^{\infty }\prod _{k=0}^{m-1} \left(1-\LL ^{2+k-m/2}t^{m}
\right)^{-1}\nonumber
\end{align}
which completes the proof of Theorem~\ref{thm_C3}.  \endproof

\begin{remark} Some formulae in the above proof appear
also in recent work of
Reineke~\cite{Reineke} and Kontsevich--Soibelman~\cite{KS_new}. In
particular, the twisted quotient \eqref{eq:twisted} appears
in~\cite[Prop.3.3]{Reineke}. The twisted quotient is applied later
in~\cite[Section 4]{Reineke} to a generating series of stacky
quotients, analogously to our series~$C$ defined
in~\eqref{eq:tildecn}. Reineke's setup is more general, dealing with
arbitrary quivers, but also more special, since there are no
relations. 
\end{remark}

\begin{remark} The result of Theorem~\ref{thm_C3} shows 
in particular that the absolute virtual motives of $\Hilb^n(\C^3)$ are 
independent of the chosen linear \CY structure on~$\C^3$. 
\end{remark} 

\begin{remark}
The first non-trivial example is the case of four points, with $\Hilb
^{4} (\C ^{3})$ irreducible and reduced but singular. The E-polynomial
realization of the virtual motive on $\Hilb ^{4} (\C ^{3})$ was
computed earlier by~\cite{Dimca-Szendroi}. The result, up to the
different normalization used there, coincides with the $t=4$ term of
the result above.
\end{remark}

\begin{remark} 
The Euler characteristic specialization of our formula is obtained by setting $\LL ^{\half }=-1$. This immediately leads to 
\[
\chi Z_{\C ^{3}} (t) = \prod _{m=1}^{\infty } (1- (-t)^{m})^{-m} = M
(-t),
\]
where $M (t)$ is the MacMahon function enumerating 3D partitions. 
The standard proof of 
\[ \chi Z_{\C^3}(t)= M(-t)\]
is by torus localization~\cite{MNOP1, Behrend-Fantechi08}. Our
argument gives a new proof of this result, which is independent of the
combinatorics of 3-dimensional partitions. Indeed, by combining the
two arguments, we obtain a new (albeit non-elementary) proof of
MacMahon's formula. It is of course conceivable that
Theorem~\ref{thm_C3} also has a proof by torus localization (perhaps
after the E-polynomial specialization). But as the computations
of~\cite{Dimca-Szendroi} show, this has to be nontrivial, since the
fixed point contributions are not pure weight.
\end{remark}

\section{The Hilbert scheme of points of a general threefold}
\label{sec: the Hilb scheme of a general threefold}

\subsection{The virtual motive of the Hilbert scheme}
\label{subsec: virtual motives of the Hilb scheme of a general threefold} Let $X$ be a
smooth and quasi-projective threefold. Recall the stratification
of $\Hilb^n(X)$ by strata $\Hilb^n_\alpha(X)$ indexed by partitions $\alpha$
of $n$. Proposition~\ref{prop_C3properties} dictates the following recipe for 
associating a virtual motive to the Hilbert scheme and its strata. 

\begin{definition}\label{defn: virtual motive of Hilb (X))}
We define virtual motives 
\[[\Hilb^n_\alpha(X)]_\vir\in\M_\C \mbox{ \ and \ } [\Hilb^n(X)]_\vir\in\M_\C\]
as follows. 
\begin{enumerate}
\item On the deepest stratum,
\[ [\Hilb_{(n)}^n(X)]_\vir = [X]\cdot [\Hilb^n(\C^3)_0]_\vir,
\label{genprop2}\]
where $[\Hilb^n(\C^3)_0]_\vir$ is as defined in~\S\ref{subsec: virtual motives for C3}.
\item More generally, on all strata, 
\[ [\Hilb^n_\alpha(X)]_\vir = \pi_{G_\alpha}\left(\left[\prod_iX^{\alpha_i}\setminus \Delta\right]\cdot \prod_i\left[\Hilb^{i}(\C^3)_0^{\alpha_i}\right]_\vir\right),
\label{genprop3}\]
where the motivic classes
$\left[\prod_iX^{\alpha_i}\setminus \Delta\right]$ and
$\prod_i\left[\Hilb^{i}(\C^d)_0)\right]^{\alpha_i}_\vir$ carry 
$G_\alpha$-actions, and $\pi_{G_\alpha}$ denotes the 
quotient map~\eqref{eq: quotient map on rings}. 
\item Finally
\[ [\Hilb^n(X)]_\vir =\sum_\alpha[\Hilb^n_\alpha(X)]_\vir.
\]
\end{enumerate}
\end{definition}

Of course by Proposition~\ref{prop_C3properties}, this definition
reconstructs the virtual motives of $\Hilb^n(\C^3)$ from those of the 
punctual Hilbert scheme $\Hilb^n(\C^3)_0$ consistently with its original definition.

\subsection{The partition function of the Hilbert scheme}\label{subsec: partition function of Hilb scheme}
Let 
\[ 
Z_X(t) = \sum_{n=0}^{\infty } [\Hilb^n(X)]_\vir t^n\in\M_\C[[t]]
\]
be the motivic degree zero \DT partition function of a smooth
quasi-projective threefold~$X$.  
We will derive expressions for this series and its
specializations from Theorem~\ref{thm_C3}.  We use the $\Exp$ map and
power structure on the \motivicring introduced in \S\ref{subsec: power
str} throughout this section.

Let 
\[
Z_{\C ^{3},0} (t) = \sum _{n=0}^{\infty }\,[\Hilb ^{n} (\C
^{3})_{0}]_{\vir }\, t^{n}
\]
be the generating series of virtual motives of the punctual Hilbert schemes of 
$\C^3$ at the origin. The following statement is the virtual
motivic analogue of Cheah's \cite[Main Theorem]{Cheah}.
 
\begin{proposition} We have
\[
Z_{X} (t) = Z_{\C ^{3},0} (t)^{[X]}.
\]
\label{proposition: universal relation}
\end{proposition}
\begin{proof} 
We have
\begin{eqnarray*} Z_X(t) & = & 1 + \sum_{\alpha} [\Hilb^n_\alpha(X)]_\vir t^{|\alpha|}\\
& = &  1 + \sum_{\alpha}\pi_{G_\alpha}\left(\left[\prod_iX^{\alpha_i}\setminus \Delta\right]\cdot \prod_i\left[\Hilb^{i}(\C^d)_0)\right]_\vir^{\alpha_i}\right) t^{|\alpha|}\\
& = & \left(1+\sum_{n\geq 1}[\Hilb^n(\C^d)_0]_\vir t^n\right)^{[X]}.
\end{eqnarray*}
Here, first we use Definition~\ref{defn: virtual motive of Hilb (X))}(3), 
then Definition~\ref{defn: virtual motive of Hilb (X))}(2), 
and finally the power structure formula~\eqref{def power structure}.
\end{proof}

\begin{theorem}\label{thm_general_formula}
Let $X$ be a smooth and quasi-projective threefold. Then
\begin{equation}\label{eq_general_formula}
Z_X(-t)= \Exp\left(\frac{-t[X]_{\vir } }{(1+\LL ^{\half }t) (1+\LL ^{-\half }t)}\right).
\end{equation}
\end{theorem}
\begin{proof} 
We begin by writing the formula from Theorem~\ref{thm_C3} using the
power structure on $\M _{\C } $ and the $\Exp $ function defined in
\S\ref{subsec: power str}.
\begin{align*}
Z_{\C ^{3}} (-t)&=\prod _{m=1}^{\infty }\prod _{k=0}^{m-1} \left(1-\LL ^{2+k-\frac{m}{2}} (-t)^{m} \right)^{-1}\\
&=\prod _{m=1}^{\infty }\prod _{k=0}^{m-1} \left(1-\left(-\LL ^{\half } \right) ^{4+2k-m} t^{m} \right)^{-1}\\
&=\prod _{m=1}^{\infty }\left(1-t^{m} \right)^{-\sum _{k=0}^{m-1}\left(-\LL ^{\half } \right) ^{4+2k-m }}\\
&=\Exp \left(\sum _{m=1}^{\infty }t^{m}\sum _{k=0}^{m-1}\left(-\LL ^{\half } \right)^{4+2k-m} \right)\\
&=\Exp \left(\sum _{m=1}^{\infty }t^{m}\left(-\LL ^{\half } \right)^{4-m} \cdot\frac{1-\left(-\LL ^{\half } \right) ^{2m}}{1-\LL } \right)\\
&=\Exp \left( \frac{\LL ^{2}}{1-\LL }\sum _{m=1}^{\infty } \left(-\LL ^{-\frac{1}{2}}t \right)^{m}-\left(-\LL ^{\frac{1}{2}}t \right)^{m}\right)\\
&=\Exp \left( \frac{\LL ^{\frac{3}{2}}}{\LL ^{-\frac{1}{2}}-\LL ^{\frac{1}{2}}}\cdot \left(\frac{-\LL ^{-\frac{1}{2}}t}{1+\LL ^{-\frac{1}{2}}t}-
\frac{-\LL ^{\frac{1}{2}}t}{1+\LL ^{\frac{1}{2}}t} \right)\right)\\
&=\Exp \left(\frac{-\LL ^{\frac{3}{2}}t}{\left(1+\LL ^{-\frac{1}{2}}t \right)\left(1+\LL ^{\frac{1}{2}}t \right)} \right).
\end{align*}
Replacing $t$ by $-t$ in Proposition~\ref{proposition: universal
relation}, and letting $X=\C ^{3}$, we find that
\[
Z_{\C ^{3},0} (-t) = \Exp \left(\frac{-\LL ^{-\frac{3}{2}}t}{\left(1+\LL ^{-\frac{1}{2}}t \right)\left(1+\LL ^{\frac{1}{2}}t \right) } \right).
\]
Another application of Proposition~\ref{proposition: universal relation} 
concludes the proof.
\end{proof}

Our formula fits nicely with the corresponding formulas for surfaces,
curves, and points. G\"ottsche's formula~\cite{Gottsche-formula,
Gottsche-motive} for a smooth quasi-projective surface $S$, rewritten
in motivic exponential form in~\cite[Statement
4]{Gusein-Zade-Luengo-Melle-Hernandez-power}, reads
\begin{equation}  \label{eq_surf}
\sum_{n=0}^\infty [\Hilb^n(S)]\,t^{n} =  \Exp \left(\frac{[S]t}{1-\LL
t}\right)
\end{equation}
and for a smooth curve $C$ we have
\[
\sum _{n=0}^{\infty } [\Hilb ^{n} (C)] t^{n} = \sum _{n=0}^{\infty } [\Sym ^{n} (C)] t^{n} = \Exp \left([C]t \right).
\]
Finally, for completeness, consider $P$, a collection of $N$
points. Then
\[
\sum _{n=0}^{\infty } [\Hilb ^{n} (P)] t^{n} = \sum
_{n=0}^{N}\binom{N}{n}t^{n} = \frac{(1-t^{2})^{N}}{(1-t)^{N}} = \Exp
\left([P]t (1-t) \right).
\]
Since $\Hilb ^{n} (X)$ is smooth and of the expected dimension when
the dimension of $X$ is 0, 1, or 2, the virtual motives are given 
by~\eqref{eqn: vir motive for smooth X}: 
\[
[\Hilb ^{n} (X)]_{\vir } = \LL ^{-\frac{n\dim X}{2}}[\Hilb ^{n} (X)].
\]
The series
\[
Z_{X} (t) = \sum _{n=0}^{\infty } [\Hilb ^{n} (X)]_{\vir }\, t^{n}
\]
is thus well defined for any $X$ of dimension 0, 1, 2, or 3. For $\dim
X \geq 4$, $Z_{X} (t)$ is defined to order $t^{3}$ since $\Hilb ^{n}
(X)$ is smooth for $n\leq 3$ in all dimensions. In order to write
$Z_{X} (t)$ as a motivic exponential, we must introduce a sign.  Let
$T= (-1)^{d}t$ where $d=\dim X$.  Then for $d$ equal to 0, 1, or 2, we
have
\begin{align*}
Z_{X} (T)&= \sum _{n=0}^{\infty } [\Hilb ^{n} (X)]\, \LL
^{-\frac{dn}{2}}T^{n}\\
&=\sum _{n=0}^{\infty } [\Hilb ^{n} (X)] \left(\left(-\LL ^{\half }
\right)^{-d} t\right)^{n}.
\end{align*}
Applying the substitution rule~\eqref{eqn: substitution rule for Exp}
to the above formulas and including the $d=3$ case from
Theorem~\ref{thm_general_formula}, we find
\[
Z_{X} (T) = \Exp \left(T[X]_{\vir }\, G_{d} (T) \right),
\]
where
\[
G_{d}(T) =\begin{cases}
1-T&d=0\\
1&d=1\\
(1-T)^{-1}&d=2\\
\left(1-\LL ^{\half }T \right)^{-1}\left(1-\LL ^{-\half }T \right)^{-1} &d=3
\end{cases}
\]

The above can be written uniformly as
\[
G_{d} (T) =\Exp \left(T[\PP ^{d-2}]_{\vir } \right),
\]
where we have defined $[\PP^{N}]_{\vir}$ for negative $N$ via the
equation
\begin{equation}\label{eqn: defn of neg dim proj space}
[\PP ^{N}]_{\vir} = \LL ^{-\frac{N}{2}}\cdot \frac{\LL ^{N+1}-1}{\LL -1}.
\end{equation}
In particular we have $[\PP ^{-1}]_{\vir }=0$ and $[\PP ^{-2}]_{\vir
}=-1$.

\begin{corollary}\label{cor: pithy formula for Z in all dims}
The motivic partition function of the Hilbert scheme of points on a
smooth variety $X$ of dimension $d$ equal to 0, 1, 2, or 3 is given by\footnote{We thank Lothar G\"ottsche, Ezra Getzler and
Sven Meinhardt for discussions which led us to this formulation.}

\[
\sum _{n=0}^{\infty }[\Hilb ^{n} (X)]_{\vir }\,\, T^n = \Exp
\Big(T\,[X]_{\vir }\, \Exp \left(T\,[\PP ^{d-2}]_{\vir } \right)\Big).
\]
where $T= (-1)^{d}t$.
\end{corollary}

\begin{remark}\label{rem: the one formula is valid for n<4} 
We do not know of a reasonable general definition for the virtual
motive $[\Hilb ^{n} (X)]_{\vir }$ when the dimension of $X$ is greater
than 3. However, it is well known that the Hilbert scheme $\Hilb^n(X)$
of $n\leq 3$ points is smooth in all dimensions and so the virtual
motive is given by $\LL ^{-\frac{nd}{2}}[\Hilb ^{n} (X)]$ in these
cases (cf.~\eqref{eqn: vir motive for smooth X}). Remarkably, the
formula in Corollary~\ref{cor: pithy formula for Z in all dims}
correctly computes the virtual motive for $n\leq 3$ in all
dimensions. This can be verified directly using the motivic class of
the punctual Hilbert scheme for $n\leq 3$ \cite[\S4]{Cheah}:
\[
\sum _{n=0}^{3} [\Hilb ^{n} (\C ^{d})_{0}]\, t^{n} = 1+t+\Lbinom{d}{1} t^{2}+\Lbinom{d+1}{2}t^{3}.
\]
\end{remark}
\begin{remark}\label{rem: Euler char specialization of the one formula is MacMahon's guess for dim d partitions}
Using the torus action on $\Hilb ^{n} (\C ^{d})$, it is easy to see
that $\chi \left(\Hilb ^{n} (\C ^{d}) \right)$ counts subschemes given
by monomial ideals. Equivalently, $\chi \left(\Hilb ^{n} (\C ^{d})
\right)$ is equal to the number of dimension $d$ partitions of
$n$. Thus naively, one expects $\chi Z_{\C ^{d}} (T)$ to be the
generating function for $d$ dimensional partitions of $n$, counted with
the sign $(-1)^{nd}$. Indeed, this is the case when $d\leq 3$ or when
$n\leq 3$. Up to the sign $(-1)^{nd}$, the Euler characteristic
specialization of our general formula yields exactly MacMahon's guess
for the generating function of dimension $d$ partitions:
\begin{align*}
\chi Z_{\C ^{d}} (T) &=\Exp \left((-1)^{d}t\,\chi [\C ^{d}]_{\vir }\, \Exp \left((-1)^{d}t\,\chi [\PP ^{d-2}]_{\vir } \right) \right)\\
&= \Exp (t\Exp ((d-1)t))\\
&=\Exp \left(\frac{t}{(1-t)^{d-1}} \right)\\
&=\prod _{m=1}^{\infty } \left(1-t^{m} \right)^{-\binom{m+d-3}{d-2}}
\end{align*}
\end{remark}
However, it is now known that MacMahon's guess is not correct,
although it does appear to be asymptotically correct in dimension
four~\cite{Mustonen-Rajesh}.

\subsection{Weight polynomial and deformed MacMahon}
When the dimension of $X$ is 1 or 2, the weight polynomial
specialization of $Z_{X} (t)$ gives rise to MacDonald's and
G\"ottsche's formulas for the Poincar\'e polynomials of the Hilbert
schemes. When the dimension of $X$ is 3, the weight polynomial
specialization leads to the following analogous formula, involving the
refined MacMahon functions discussed in Appendix~\ref{appendix: q-def
of MacMahon}.

\begin{theorem}\label{thm: formula for virtual weight polys} 
Let $X$ be a smooth projective threefold and let $b_{d}$ be the Betti
number of $X$ of degree $d$. Then the generating function of the
virtual weight polynomials of the Hilbert schemes of points of~$X$ is
given by
\begin{equation}\label{eq_prod_Mac}
WZ_X(t) = \prod_{d=0}^{6} M_{\frac{d-3}{2}}\left(-t,-q^{\frac{1}{2}} \right)^{(-1)^{d} b_d}\in\Z[q^{\pm\half}][[t]],
\end{equation} 
where 
\[
M_{\delta}(t,q^{\frac{1}{2}}) = \prod _{m=1}^{\infty }\prod
_{k=0}^{m-1}\left(1-q^{\delta +\frac{1}{2}+k-\frac{m}{2}}\,t^{m}
\right)^{-1}
\]
are the refined MacMahon functions discussed in
Appendix~\ref{appendix: q-def of MacMahon}.
\end{theorem}

\proof Recall that the weight polynomial specialization 
\[
W:\M _{\C
}\to \Z\left[q^{\pm \frac{1}{2}} \right]
\]
is obtained from the $E$
polynomial specialization
\[
E:\M _{\C } \to  \Z \left[x,y,(xy)^{-\half } \right]
\]
by setting $x=y=-q^{\frac{1}{2}}$ and $(xy)^{\half }=q^{\half }$.  It
follows from \cite[Prop~4]{Gusein-Zade-Luengo-Melle-Hernandez-Hilb}
that the $W$-specialization is a ring homomorphism which
respects power structures where the power structure on $\Z
\left[q^{\pm \half }\right]$ satisfies
\[
(1-t^{m})^{-\sum _{i}a_{i}\left(-q^{\half } \right)^{i}} = \prod _{i}
\left(1- \left(-q^{\half } \right)^{i}t^{m} \right)^{-a_{i}}.
\]

From Theorems~\ref{thm_general_formula}  we deduce that 
\[
Z_{X} (t) = Z_{\C ^{3}} (t)^{\LL ^{\frac{3}{2}}[X]_{\vir }} = Z_{\C ^{3}} (t)^{[X]}.
\]
It then follows from \ref{thm_C3} that
\[
Z_{X} (t) = \prod _{m=1}^{\infty }\prod _{k=0}^{m-1}\left(1-\left(-\LL ^{\half } \right)^{2k+2-m} (-t)^{m} \right)^{-[X]}.
\]
Applying the homomorphism $W$ to $Z_{X}$, using the compatibility
of the power structures, we get
\begin{align*}
WZ_{X} (t) &= \prod _{m=1}^{\infty }\prod _{k=0}^{m-1}\left(1-\left(-q ^{\half } \right)^{2k+2-m} (-t)^{m} \right)^{-\sum _{d=0}^{6} (-1)^{d} b_{d}\left(-q^{\half } \right)^{d}}\\
&=\prod _{d=0}^{6}\prod _{m=1}^{\infty }\prod _{k=0}^{m-1}\left(1-\left(-q ^{\half } \right)^{2k+2-m+d} (-t)^{m} \right)^{-(-1)^{d} b_{d}}\\
&=\prod _{d=0}^{6} M_{\frac{d-3}{2}} \left(-t,-q^{\half } \right)^{(-1)^{d}b_{d}}.
\end{align*}
\endproof

\begin{remark}\label{rem:CY_or_not_CY}
The Euler characteristic specialization is easily determined from the
formula in Theorem~\ref{thm: formula for virtual weight polys} by setting $-q^{\half }=1$, namely
\begin{equation} \label{eq_Eu_Hilb}
\chi Z_X(t) = M(-t)^{\chi(X)}.
\end{equation}
By Proposition~\ref{prop_spec_to_kai}, this is the partition function
of the ordinary degree zero \DT invariants in the case when $X$ is
Calabi--Yau. Formula~\eqref{eq_Eu_Hilb} is a result, for any smooth
quasi-projective threefold, of Behrend and
Fantechi~\cite{Behrend-Fantechi08}.

Note that a variant of this formula, for a smooth projective threefold
$X$, was originally conjectured by Maulik, Nekrasov, Okounkov and
Pandharipande~\cite{MNOP1}. This involves the integral (degree) of the
degree zero virtual cycle on the Hilbert scheme, and says that for a
projective threefold~$X$,
\begin{equation}\label{eq_MNOP}
\sum_{n=0}^\infty \deg[\Hilb^n(X)]^\vir t^n = M(-t)^{\int_X c_3-c_1c_2},
\end{equation}
with $c_i$ being the Chern classes of $X$.  This was proved by
\cite{Levine-Pandharipande, Li-zeroDT}. The two formulae become
identical in the projective \CY case, since then the perfect
obstruction theory on the Hilbert schemes is
symmetric~\cite{Behrend-Fantechi08}, and so by the main result
of~\cite{Behrend-micro}, the degree of the virtual zero-cycle is equal
to its virtual Euler characteristic. Note that our work has nothing to
say about formula~\eqref{eq_MNOP} in the non-\CY case.
\end{remark}

\subsection{Categorified \DT invariants}\label{subsec: categorified DT
invs} Our definition of the virtual motive of the Hilbert scheme
$\Hilb^n(X)$ of a smooth quasi-projective threefold $X$, obtained by
building it up from pieces on strata, is certainly not ideal. Our
original aim in this project was in fact to build a categorification
of \DT theory on the Hilbert scheme of a (Calabi--Yau) threefold~$X$,
defining an object of some category with a cohomological functor to
(multi)graded vector spaces, whose Euler characteristic gives the
degree zero \DT invariant of~$X$. Finding a ring with an Euler
characteristic homomorphism is only a further shadow of such a
categorification.

One particular candidate where such a categorification could live
would be the category $\MHM(\Hilb^n(X))$ of mixed Hodge
modules~\cite{Saito-MHP, Saito-MHM} on the Hilbert scheme. This
certainly works for affine space~$\C^3$, since the global description
of the Hilbert scheme $\Hilb^n(\C^3)$ as a degeneracy locus gives rise
to a mixed Hodge module of vanishing cycles, with all the right
properties~\cite{Dimca-Szendroi}.  However, when trying to globalize
this construction, we ran into glueing issues which we couldn't
resolve, arising from the fact that the description of $\Hilb^n(\C^3)$
as a degeneracy locus uses a linear \CY structure on~$\C^3$
and is therefore not completely canonical.

In some particular cases, we were able to construct the mixed Hodge
modules categorifying \DT theory of the Hilbert scheme.  Since we
currently have no application for categorification as opposed to a
refined invariant taking values in the \motivicring, and since our
results are partial, we only sketch the constructions.

\begin{itemize}
\item {\em Low number of points.}\ For $n\leq 3$, the Hilbert scheme
$\Hilb^n(X)$ is smooth and there is nothing to do. The next case $n=4$
is already interesting. It is known that for a threefold $X$, the
space $\Hilb^4(X)$ is irreducible and reduced, singular along a copy
of $X$ which is the locus of squares of maximal ideals of points.  As
proved in~\cite{Dimca-Szendroi}, for $X=\C^3$ the mixed Hodge module
of vanishing cycles of $f_4$ on $\Hilb^4(\C^3)$ admits a very natural
geometric description: it has a three-step non-split filtration with
quotients being (shifted copies of) the constant sheaf on the (smooth)
singular locus, the intersection cohomology (IC) sheaf of the whole
irreducible space $\Hilb^4(\C^3)$, and once more the constant sheaf on
the singular locus. It also follows from results of [ibid.]~that the
relevant extension groups are one-dimensional, and so this mixed Hodge
module is unique. Turning to a general (simply-connected)~$X$, we
again have the IC sheaf on the space $\Hilb^4(X)$ and the constant
sheaf on its singular locus, and a compatible extension of these mixed
Hodge modules exists and is unique. This provides the required
categorification.  We expect that such an explicit construction is
possible for some higher values of~$n$ than $4$ but certainly not in
general.
\item {\em Abelian threefolds.}\ Let $X$ be an abelian threefold (or
some other quotient of $\C^3$ by a group of translations). Then we can
cover $X$ by local analytic patches with transition maps which are in
the affine linear group of~$\C^3$. The local (analytic) vanishing
cycle sheaves on the Hilbert schemes of patches can be glued using the
affine linear transition maps to a global (analytic) mixed Hodge
module on $\Hilb^n(X)$.
\item {\em Local toric threefolds.}\ Finally, it should be possible to
construct the gluing directly for some local toric threefolds. We
checked the case of local~$\PP^1$ explicitly, in which case the mixed
Hodge modules on all Hilbert schemes exist. However, we already failed
for local~$\PP^2$.
\end{itemize}

Compare also the discussion surrounding~\cite[Question 5.5]{Joyce-Song},
and see also the recent work~\cite{KS_new}.

\section*{Acknowledgements} We would like to thank D. Abramovich,
T. Bridgeland, P. Brosnan, A. Dimca, B. Fantechi, E. Getzler,
L. G\"ottsche, I. Grojnowski, D. Joyce, T. Hausel, F. Heinloth,
S. Katz, M. Kontsevich, S. Kov\'acs, E. Looijenga, S. Meinhardt,
G. Moore, A. Morrison, J. Nicaise, R. Pandharipande, A. Rechnitzer,
R. Thomas, M. Saito, J. Sch\"urmann, Y. Soibelman and D. van Straten
for interest in our work, comments, conversations and helpful
correspondence. Some of the ideas of the paper were conceived during
our stay at MSRI, Berkeley, during the Jumbo Algebraic Geometry
Program in Spring 2009; we would like to thank for the warm
hospitality and excellent working conditions there. JB thanks the
Miller Institute and the Killiam Trust for support during his
sabbatical stay in Berkeley. BS's research was partially supported by
OTKA grant K61116.

\appendix \section{$q$-deformations of the MacMahon
function}\label{appendix: q-def of MacMahon}

Let $\CP$ denote the set of all finite 3-dimensional partitions. For a
partition $\alpha\in\CP$, let $w(\alpha)$ denote the number of boxes
in~$\alpha$. The combinatorial generating series
\[ M(t) = \sum_{\alpha\in\CP} t^{w(\alpha)}
\]
was determined in closed form by MacMahon~\cite{MacMahon} to be
\[ M(t)= \prod_{m=1}^\infty (1-t^m)^{-m}.
\]
Motivated by work of Okounkov and
Reshetikhin~\cite{Okounkov-Reshe-skew}, in a recent
paper~\cite{Iqbal-Kozcaz-Vafa}, Iqbal--Koz\c{c}az--Vafa discussed a
family of $q$-deformations of this formula. Think of a 3-dimensional
partition~$\alpha\in\CP$ as a subset of the positive octant lattice
$\N^3$, and break the symmetry by choosing one of the coordinate
directions.  Define $w_-(\alpha), w_0(\alpha)$ and $w_+(\alpha)$,
respectively, as the number of boxes (lattice points) in
$\alpha\cap\{x-y<0\}$, $\alpha\cap\{x-y=0\}$ and
$\alpha\cap\{x-y>0\}$. For a half-integer $\delta\in\half\Z$, consider
the generating series
\[ M_\delta(t_1, t_2) = \sum_{\alpha\in\CP} t_1^{w_-(\alpha) + \left(\frac{1}{2}+\delta\right)w_0(\alpha)}t_2^{w_+(\alpha) + \left(\frac{1}{2}-\delta\right)w_0(\alpha)}.
\]
Clearly $M_\delta(t,t)=M(t)$ for all $\delta$. 

\begin{theorem} 
{\rm (Okounkov--Reshetikhin~\cite[Thm. 2]{Okounkov-Reshe-skew})}
The series $M_\delta(t_1, t_2)$ admits the product form
\[
M_\delta(t_1, t_2)=\prod_{i,j=1}^\infty \left(1-t_1^{i-\frac{1}{2}+\delta}t_2^{j-\frac{1}{2}-\delta}\right)^{-1}.
\]  
\end{theorem}

In the main body of the paper, we use a different set of
variables. Namely, we set
\[
t_{1} = tq^{\half },\quad t_{2} = tq^{-\half }.
\]
Then the product formula becomes
\[
M_\delta(t,q^\half)=\prod_{m=1}^\infty \prod_{k=0}^{m-1} \left(1-t^mq^{k+\frac{1}{2}-\frac{m}{2}+\delta}\right)^{-1}.
\]
The specialization to the MacMahon function is
$M_\delta(t,q^\half=1)=M(t)$ for all~$\delta$.

\section{The motivic nearby fiber of an equivariant
function}\label{appendix: nearby fiber proof}

In this appendix we prove Proposition~\ref{prop: if f is C* equiv and
circle compact, then vanishing cycle is X1-X0}, which asserts that if a
regular function $f:X\to \C $ on a smooth variety is
equivariant with respect to a torus action satisfying certain
assumptions, then Denef-Loeser's motivic nearby fiber $[\psi _{f}]$
is simply equal to the motivic class of the geometric fiber $[f^{-1}
(1)]$. To make this appendix self-contained, we recall the definitions
and restate the result below.

Let 
\[
f\colon X\to\C
\]
be a regular function on a smooth quasi-projective variety~$X$, and
let $X_0=f^{-1}(0)$ be the central fiber.  Denef and Loeser define
$[\psi _{f}]\in \M ^{\muhat }_{\C }$, the motivic nearby cycle
of $f$ using arc spaces and the motivic zeta function
\cite{Denef-Loeser-Igusa,Looijenga-motivic}. Using motivic
integration, they give an explicit formula for $[\psi _{f}]$ in terms
of any embedded resolution which we now recall.

Let $h:Y\to X$ be an embedded resolution of $X_{0}$, namely~$Y$ is
non-singular and
\[
\tilde{f}=h\circ f:Y\to \C
\]
has central fiber $Y_0$ which is a normal crossing divisor with
non-singular components $\{E_{j}\, :\, j\in J\}$. For $I\subset J$,
let
\[
E_{I} = \bigcap _{i\in I} E_{i}
\]
and let
\[
E_{I}^{o} = E_{I} - \bigcup _{j\in I^{c}} \left(E_{j}\cap E_{I} \right).
\]
By convention, $E_{\emptyset }=Y$ and $E^{o}_{\emptyset }=Y-Y_{0}$.

Let $N_{i}$ be the multiplicity of $E_{i}$ in the divisor
$\tilde{f}^{-1} (0)$. Letting \[m_{I} = gcd (N_{i})_{i\in I},\]
there is a natural etale cyclic $\mu_{m_I}$-cover
\[
\tilde{E}^{o}_{I}\to E^{o}_{I}.
\]
The formula of Denef and Loeser for the (absolute) motivic nearby
cycles of~$f$ is given by
\begin{equation}\label{eqn: Denef-Loeser defn of motive of nearby cycles}
[\psi_f]=\sum _{I \neq \emptyset} (1-\LL )^{|I|-1}[\tilde
E^{o}_{I},\mu_{m_{I}}]\quad \in \M _{\C }^{\muhat }
\end{equation}

Since all the $E_{I}^{o}$ appearing in the above sum have natural maps
to $X_{0}$, the above formula determines the \emph{relative} motivic
nearby cycle $[\psi _{f}]_{X_{0}}\in \M ^{\muhat }_{X_{0}}$.  The
relative motivic vanishing cycle is supported on $Z=\{df=0 \}$, the
degeneracy locus of $f$:
\[
[\phi _{f}]_{Z} = [\psi _{f}]_{X_{0}}-[X_{0}]_{X_{0}} \in \M
_{Z}^{\muhat } \subset \M _{X_{0}}^{\muhat }.
\]

Recall that an action of $\C^{*}$ on a variety $V$ is \emph{circle
compact}, if the fixed point set $V^{\C^{*}}$ is compact and moreover,
for all $v\in V$, the limit $\lim_{t\to 0}t \cdot y$ exists.

The following is a restatement of Proposition~\ref{prop: if f is C*
equiv and circle compact, then vanishing cycle is X1-X0} and
Proposition~\ref{prop: relative and effective version of vanishing
cycle = X1 - X0 proposition}.  It is the main result of this appendix.
\begin{theorem}\label{thm: main result of appendix}
Let $f:X\to \C$ be a regular morphism on a smooth quasi-projective
complex variety. Let $Z=\{df=0 \}$ be the degeneracy locus of $f$ and
let $Z_{\aff }\subset X_{\aff }$ be the affinization of $Z$ and $X$
respectively. Assume that there exists an action of a connected
complex torus $T$ on $X$ so that $f$ is $T$-equivariant with respect
to a primitive character $\chi :T\to \C ^{*}$, namely $ f (t\cdot
x)=\chi (t) f (x)$ for all $x\in X$ and $t\in T$. We further assume
that there exists a one parameter subgroup $\C ^{*}\subset T$ such
that the induced action is circle compact.  Then
the motivic nearby cycle class $[\psi_{f}]$ is in $\M
_{\C}\subset \M ^{\muhat }_{\C}$ and is equal to $[X_{1}]=[f^{-1}
(1)]$. Consequently the motivic vanishing cycle class $[\phi _{f}]$ is
given by
\[
[\phi _{f}] = [f^{-1} (1)] - [f^{-1} (0)].
\]
If we further assume that $X_{0}$ is reduced then $[\phi
_{f}]_{Z_{\aff }}$, the motivic vanishing cycle, considered as a
relative class on $Z_{\aff }$, lies in the subring $\M _{Z_{\aff
}}\subset \M _{Z_{\aff }}^{\muhat }$.
\end{theorem}

By equivariant resolution of singularities
\cite[Cor~7.6.3]{Villamayor}, we may assume that $h:Y\to X$, the
embedded resolution of $X_{0}$, is $T$-equivariant.  Namely $Y$ is a
non-singular $T$-variety and
\[
\tilde{f}=h\circ f:Y\to \C
\]
is $T$-equivariant with central fiber $E$ which is a normal crossing
divisor with non-singular components $E_{j}$, $j\in J$. Let
$\C^{*}\subset T$ be the one-parameter subgroup whose action on $X$ is
circle compact. Then the action of $\C^{*} $ on $Y$ is circle compact
(since $h$ is proper) and each $E_{j}$ is invariant (but not
necessarily fixed).

We will make use of the \BiBi decomposition for smooth varieties
\cite{Bialynicki-Birula}. \BiBi proves that if $V$ is a smooth
projective variety with a $\C ^{*} $-action, then there is a locally
closed stratification:
\[
V=\bigcup _{F}Z_{F}
\]
where the union is over the components of the fixed point locus and
$Z_{F}\to F$ is a Zariski locally trivial affine bundle. The rank of
the affine bundle $Z_{F}\to F$ is given by
\[
n (F) = \Index (N_{F/V})
\]
where the index of the normal bundle $N_{F/V}$ is the number of
positive weights of the fiberwise action of $\C^{*}$. The morphisms
$Z_{F}\to F$ are defined by $x\mapsto \lim_{t\to 0}t\cdot x$ and
consequently, the above stratification also exists for smooth
varieties with a circle compact action. As a corollary of the \BiBi
decomposition, we get the following relation in the \motivicring.

\begin{lemma}\label{lem: BB for circle compact actions}
Let $V$ be a smooth quasi-projective variety with a circle compact
$\C^{*}$-action. For each component $F$ of the fixed point locus,
we define the index of $F$, denoted by $n (F)$, to be the number of
positive weights in the action of $\C^{*}$ on $N_{F/V}$. Then in
$\M_{\C}$ we have
\[
[V] = \sum _{F}\, \LL  ^{n (F)}\, [F],
\]
where the sum is over the components of the fixed point locus and 
$\LL =[\mathbb{A}^{1}_{\C}]$ is the Lefschetz motive.
\end{lemma}
We call this decomposition a \emph{BB decomposition}.

We begin our proof of Theorem~\ref{thm: main result of appendix} with
a ``no monodromy'' result.
\begin{lemma}\label{lem: no monodromy}
The following equation holds in $\M _{\C }^{\muhat }$:
\[
[\tilde{E}^{o}_{I},\mu _{m_{I}}] = [E_{I}^{o}].
\]
Under the further assumption that $X_{0}$ is reduced,
the above equation holds in $\M _{X_{\aff }}^{\muhat }$.
\end{lemma}
An immediate corollary of Lemma~\ref{lem: no monodromy} is that $[\psi
_{f}]$ lies in the subring $\M _{\C }\subset \M _{\C }^{\muhat }$ and
that if $X_{0}$ is reduced, then $[\phi _{f}]_{Z_{\aff }}$ lies in the
subring $\M _{Z_{\aff }}\subset \M _{Z_{\aff }}^{\muhat }$.

To prove Lemma~\ref{lem: no monodromy}, we recall the construction of
the $\mu _{m_{I}}$-cover $\tilde {E}_{I}^{o}\to E^{o}_{I}$ given in
\cite[\S~5]{Looijenga-motivic}.  Let $N=lcm (N_{i})$ and let $\tilde
{Y}\to Y$ be the $\mu _{N}$-cover obtained by base change over the
$N$th power map $(\cdot )^{N}:\C \to \C $ followed by
normalization. Define $\tilde {E}^{o}_{I}$ to be any connected
component of the preimage of $E^{o}_{I}$ in $\tilde {Y}$. The
component $\tilde {E}^{o}_{I}$ is stabilized by $\mu _{m_{I}}\subset
\mu _{N}$ whose action defines the cover $\tilde {E}^{o}_{I}\to
E^{o}_{I}$.

Observe that the composition
\[
E_{i} \to Y \to X \to  X_{\aff }
\] 
contracts $E_{i}$ to a point unless $E_{i}$ is a component of the
proper transform of $X_{0}$. Thus for these components (and their
intersections), the proof given below (stated for absolute classes),
applies to relative classes over $X_{\aff }$ as well. Under the
assumption that $X_{0}$ is reduced, those $E_{i}$ which are components
of the proper transform of $X_{0}$ have multiplicity one and so
$\tilde {E}_{i}=E_{i}$ and there is nothing to prove.

We define $\tilde {T}$ by the fibered product
\[
\begin{diagram}[height=1.8em,width=1.8em,nohug]
\tilde {T}&\rTo ^{\tilde {\chi }} & \C ^{*}\\
\dTo &&\dTo_{(\cdot )^{N}}\\
{T}&\rTo ^{ {\chi }} & \C ^{*}\\
\end{diagram}
\]
Thus $\tilde {T}$ is an extension of $T$ by $\mu _{N}\subset \C ^{*}$
and it has character $\tilde {\chi }$ satisfying $\tilde {\chi
}^{N}=\chi $. Moreover, $\tilde {\chi }$ is the identity on the
subgroup $\mu _{N}\subset \tilde {T}$. The key fact here is that since
$\chi $ is primitive, $\tilde {T}$ is connected.

By construction, $\tilde {T}$ acts on the base change of $Y$ over the
$N$th power map and hence it acts on the normalization $\tilde
{Y}$. Thus we have obtained an action of a connected torus $\tilde
{T}$ on $\tilde {Y}$ covering the $T$-action on $Y$. Since $T$ acts on
each $E^{o}_{I}$, $\tilde {T}$ acts on each component of the preimage
of $E^{o}_{I}$ in $\tilde {Y}$. Thus we have an action of the
connected torus $\tilde {T}$ on $\tilde {E}^{o}_{I}$ such that the
$\mu _{m_{I}}$-action is induced by the subgroup 
\[
\mu _{m_{I}}\subset
\mu _{N}\subset \tilde {T}.
\]
Lemma~\ref{lem: no monodromy} then
follows from the following:

\begin{lemma}\label{lemma cyclic subgroup}
Let $W$ be a smooth quasi-projective variety with the action of a
connected torus $\tilde {T}$. Then for any finite cyclic subgroup $\mu
\subset \tilde {T}$, the equation
\[
[W,\mu ] = [W] = [W/\mu ]
\]
holds in $\M ^{\muhat }_{\C  }$. 
\end{lemma}

\proof Let $\C ^{*}\subset \tilde {T}$ be the 1-parameter subgroup
generated by $\mu \subset \tilde {T}$. The $\C ^{*} $-action on $W$
gives rise to a $\C ^{*}$-equivariant stratification of $W$ into
varieties $W_{i}$ of the form $\left(V-\{0 \} \right)\times F$ where
$F$ is fixed and $V$ is a $\C ^{*}$-representation. This assertion
follows from applying the \BiBi decomposition to $\overline{W}$, any
$\C ^{*}$-equivariant smooth compactification of $W$ and stratifying
further to trivialize all the bundles and to make the zero sections
separate strata. The induced stratification of $W$ then is of the
desired form. Thus to prove Lemma~\ref{lemma cyclic subgroup}, 
it then suffices to prove it
for the case of $\mu$ acting on $V-\{0\}$ where $V$ is a 
$\mu$-representation. By the relation given in 
equation~\eqref{eqn: vector bundle with G action is trivial in equivariant K-group}, 
we have $[V,\mu ]=[V]$ and hence $[V-\{0 \},\mu ] = [V-\{0 \}]$. 
The equality
$[V-\{0 \}] = [\left(V-\{0 \} \right)/\mu ]$ follows from
\cite[Lemma~5.1]{Looijenga-motivic}.\qed

Applying Lemma~\ref{lemma cyclic subgroup}
to~\eqref{eqn: Denef-Loeser defn of motive of nearby cycles}, 
we see that in order to prove $[\psi_{f}]=[X_{1}]$, we must prove
\[
[X_{1}] = \sum _{I \neq \emptyset } (1-\LL )^{|I|-1} [E^{o}_{I}].
\]

As explained in the beginning of \S\ref{subsec:triv}, there is an
isomorphism
\[
X_{1}\times \C ^{*} \cong X-X_{0}.
\]
Consequently we get $(\LL -1)[X_{1}] = [X]-[X_{0}]$ or
equivalently
\[
[Y] = [Y_{0}] + (\LL  -1) [X_{1}].
\]
Combining this with the previous equation we find that the equation we
wish to prove, $[\psi_{f}] = [X_{1}]$, is equivalent to
\[
[Y] = [Y_{0}] - \sum _{I\neq \emptyset } (1-\LL  )^{|I|} [E_{I}^{o}],
\]
which can also be written (using the conventions about the empty set)
as
\begin{equation}\label{eqn: thm reformulation 1}
0=\sum _{I} (1-\LL  )^{|I|} [E_{I}^{o}].
\end{equation}
By the principle of inclusion/exclusion, we can write
\begin{align*}
[E_{I}^{o}] &= [E_{I}] -\sum _{\emptyset \neq K\subset I^{c}} (-1)^{|K|} [E_{I\cup K}]\\
&= \sum _{K\subset I^{c}} (-1)^{K}[E_{I\cup K}],
\end{align*}
thus we have 
\[
(-1)^{|I|}[E^{o}_{I}] = \sum _{A\supset I} (-1)^{|A|}[E_{A}].
\]
Thus the right hand side of equation~\eqref{eqn: thm reformulation 1}
becomes
\[
\sum _{I} (\LL  -1)^{|I|} \sum _{A\supset I} (-1)^{|A|} [E_{A}] = \sum
_{A} (-1)^{|A|} [E_{A}] \sum _{I\subset A} (\LL  -1)^{|I|}.
\]
Since
\[
\sum _{I\subset A} (\LL  -1)^{|I|} =\sum _{n=0}^{|A|} \binom{|A|}{n} (\LL 
-1)^{n} = \LL  ^{|A|},
\]
we can reformulate the equation we need to prove as
\begin{equation}\label{eqn: thm reformulation 2}
0 = \sum _{A} (-\LL  )^{|A|} [E_{A}].
\end{equation}

Note that each $E_{A}$ is smooth and $\C^{*}$ invariant and that
the induced $\C^{*}$-action on $E_{A}$ is circle compact, so we
have a BB decomposition for each.

Let $F$ denote a component of the fixed point set $Y^{\C^{*}}$
and let $\theta _{F,A}$ denote a component of $F\cap
E_{A}$. Since $E_{A}$ is smooth and $\C^{*}$
invariant and the induced $\C^{*}$-action is circle compact,
$E_{A}$ admits a BB decomposition
\[
[E_{A}] = \sum _{F}\sum _{\theta _{F,A}} \LL  ^{n (\theta _{F,A})}
[\theta _{F,a}]
\]
where 
\[
n (\theta _{F,A}) = \Index  (N_{\theta _{F,A}/E_{A}}) .
\]
Therefore the sum in equation~\eqref{eqn: thm reformulation 2} (which
we wish to prove is zero) is given by
\[
\sum _{A} (-\LL  )^{|A| } [E_{A}] = \sum _{A}\sum _{F}\sum _{\theta
_{F,A}} (-1)^{|A|}\,\,\LL  ^{|A|+n (\theta _{F,A})}\,\, [ \theta _{F,A}].
\]
We define a set $I (F)$ by
\[
I (F) = \{i:\,\, F\subset E_{i} \}.
\]
Then clearly 
\[
\theta _{F,A} = \theta _{F,A'}\quad \text{if}\quad A\cup I (F) =A'\cup
I (F).
\]
So writing $A=B\cup C$ where $B\subset I (F)$ and $C\subset I
(F)^{c}$, we can rewrite the above sum as
\begin{equation}\label{eqn: sum with BB decomp}
\sum _{F}\,\,\sum _{C\subset I (F)^{c}}\,\,\sum _{\theta _{F,C}} (-\LL  )^{|C|} \,[\theta _{F,C}]\,\sum _{B\subset I (F)} (-1)^{|B|} \LL  ^{|B|+n (\theta _{F,B\cup C})}\,.
\end{equation}
We will show that the inner most sum is always zero which will prove
Theorem~\ref{thm: main result of appendix}.

Let $y\in \theta _{F,A}$ where $A=B\cup C$. We write the $\C^{*}
$ representation $T_{y}Y$ in two ways:
\begin{align*}
TF + N_{F/Y} &= TE_{A} + N_{E_{A}/Y} \\
&= T\theta _{F,A} + N_{\theta _{F,A}/E_{A}}+\sum _{i\in A}N_{E_{i}/Y}
\end{align*}
where restriction to the point $y$ is implicit in the above equation.
Counting positive weights on each side, we get
\[
n (F) = n (\theta _{F,A}) +\sum _{i\in A} m_{i} (\theta _{F,A})
\]
where
\[
m_{i} (\theta _{F,A}) = \Index (N_{E_{i}/Y}|_{\theta _{F,A}}).
\]
Note that $m_{i} (\theta _{F,A}) $ is 1 or 0 depending on if the
weight of the $\C^{*}$-action on $E_{i}|y$ is positive or
not. Note also that if $i\in I (F)$ then
\[
m_{i} (\theta _{F,A}) = m_{i,F} = \Index (N_{E_{i}/Y}|_{F}).
\]
Moreover, if $i\in I (F)^{c}$, then 
\[
N_{E_{i}/Y}|_{y} \subset  TF|_{y}
\]
and so $m_{i} (\theta _{F,A})=0$. 

Thus writing $A=B\cup C$ with $B\subset I (F)$ and $C\subset I (F)^{c}$, 
we get 
\[
n (\theta _{F,A}) = n (F) - \sum _{i\in B} m_{i,F},
\]
and so
\[
\sum _{B\subset I (F)} (-1)^{|B|}\,\LL  ^{|B|+n (\theta _{F,A})} = \LL 
^{n (F)}\,\sum _{B\subset I (F)} (-1)^{|B|}\, \LL  ^{\sum _{i\in B}
(1-m_{i,F})}.
\]
For $k=0,1$, we define
\[
I_{k} (F) = \{i\in I (F):\quad m_{i,F}=k \}.
\]
Then
\[
\sum _{B\subset I (F)} (-1)^{|B|}\, \LL  ^{\sum _{i\in B} (1-m_{i,F})} =
\sum _{B_{0}\subset I_{0} (F)} (-\LL  )^{|B_{0}|} \sum _{B_{1}\subset
I_{1} (F)} (-1)^{|B_{1}|},
\]
but 
\[
\sum _{B_{1}\subset I_{1} (F)} (-1)^{|B_{1}|} = 0
\]
unless $I_{1} (F)=\emptyset $. Since the above equation implies that
the expression in equation~\eqref{eqn: sum with BB decomp} is zero,
and that in turn verifies equations~\eqref{eqn: thm reformulation 1}
and \eqref{eqn: thm reformulation 2} which are equivalent to
Theorem~\ref{thm: main result of appendix}, 
it only remains for us to prove that 
$I_{1} (F)\neq \emptyset$ for all $F$.

Let $y\in F$. We need to show that for some $i$, the action of $\C
^{*} $ on $N_{E_{i}/Y}|_{y}$ has positive weight.

By the Luna slice theorem, there is an etale local neighborhood of
$y\in Y$ which is equivariantly isomorphic to $T_{y}Y$. Over the point
$y$, we have a decomposition
\[
TY = TF + N_{F/E_{I (F)}} +\sum _{i\in I (F)} N_{E_{i}/Y}.
\]
Let $(u_{1},\dots ,u_{s},v_{1},\dots ,v_{p},\{w_{i} \}_{i\in I (F)})$
be linear coordinates on $T_{y}Y$ compatible with the above
splitting. The action of $t\in \C^{*}$ on $T_{y}Y$ is given by
\[
t\cdot (u,v,w) = (u_{1},\dots ,u_{s},t^{a_{1}}v_{1},\dots ,t^{a_{p}}v_{p},\{t^{b_{i}}w_{i} \}_{i\in I (F)}).
\]
In these coordinates, the function $\tilde{f}$ is given by
\[
\tilde{f} (u,v,w) = g (u) \prod _{i\in I (F)} w_{i}^{N_{i}}
\]
where $g (u)$ is a unit.  Since the $\C ^{*}$-action on $Y$ is circle
compact, $\tilde {f}$ is equivariant with respect to an action on $\C$
of positive weight $l$, that is
\[
\tilde{f} (t\cdot (u,v,w)) = t^{l} \tilde{f} (u,v,w).
\]
This implies that 
\[
l=\sum _{i\in I (F)} b_{i}N_{i}.
\]
Then since $N_{i}>0 $ and $l>0$ we have that $b_{i}>0$ for some $i\in
I (F)$ and so for this $i$, we have $m_{i,F}=1$ which was what we
needed to prove. The proof of Theorem~\ref{thm: main
result of appendix} is now complete.\qed

\bibliography{mainbiblio}
\bibliographystyle{plain}

\end{document}